\documentclass[11pt,reqno]{amsproc}
\allowdisplaybreaks

\usepackage{amssymb}
\usepackage{stmaryrd}

\usepackage{enumerate}
\usepackage{fullpage}
\usepackage{textcomp}
\usepackage{lettrine}

\usepackage{graphicx}
\usepackage{subfigure}
\usepackage{morefloats}

\usepackage[debug=false, colorlinks=true, pdfstartview=FitV,
linkcolor=blue, citecolor=blue, urlcolor=blue]{hyperref}
\usepackage[semicolon,square,authoryear,sort]{natbib}

\usepackage[doipre={DOI:~}]{uri}

\usepackage{color}
\usepackage{colortbl}
\usepackage[usenames,dvipsnames,table]{xcolor}
\usepackage[most]{tcolorbox}

\newlength{\drop}
\definecolor{amethyst}{rgb}{0.6, 0.4, 0.8}
\definecolor{burgundy}{rgb}{0.5, 0.0, 0.13}

\usepackage{longtable} 
\usepackage{multirow} 
\usepackage{rotating} 
\usepackage{bigstrut} 
\usepackage{hhline} 

\usepackage{amsthm}

\newenvironment{proofidea}{%
  \proof}{\endproof}
  
\newenvironment{layperson}{%
  \proof}{\endproof}

\newtheoremstyle{remboldstyle}
  {}{}{}{}{\bfseries}{.}{.5em}{{\thmname{#1 }}{\thmnumber{#2}}{\thmnote{ (#3)}}}
\theoremstyle{remboldstyle}

\newtheorem{theorem}{Theorem}[section]

\newtheorem{proposition}[theorem]{Proposition}
\newtheorem{corollary}[theorem]{Corollary}

\title{\textbf{Modeling thermal regulation in thin vascular systems: 
A mathematical analysis}}

\author{\textbf{Kalyana B.~Nakshatrala} \\
  {\small Associate Professor, Department of Civil and Environmental Engineering \\
  University of Houston, Houston, Texas 77204, USA.}\\
  {\small email:~\texttt{knakshatrala@uh.edu}, phone: +1-713-743-4418}}

\keywords{thermal regulation; vascular systems; reduced-order modeling; maximum and comparison principles; mathematical analysis; efficiency}

\begin{document}

\begin{titlepage}
  \drop=0.1\textheight
  \centering
  \vspace*{\baselineskip}
  \rule{\textwidth}{1.6pt}\vspace*{-\baselineskip}\vspace*{2pt}
  \rule{\textwidth}{0.4pt}\\[\baselineskip]
       {\Large \textbf{\color{burgundy}
       Modeling thermal regulation in thin vascular systems: \\[0.3\baselineskip]
       A mathematical analysis}}\\[0.3\baselineskip]
       \rule{\textwidth}{0.4pt}\vspace*{-\baselineskip}\vspace{3.2pt}
       \rule{\textwidth}{1.6pt}\\[\baselineskip]
       \scshape
       An e-print of this paper is available on arXiv. \par
       \vspace*{1\baselineskip}
       Authored by \\[\baselineskip]

  {\Large K.~B.~Nakshatrala\par}
  {\itshape Department of Civil \& Environmental Engineering \\
  University of Houston, Houston, Texas 77204. \\
  \textbf{phone:} +1-713-743-4418, \textbf{e-mail:} knakshatrala@uh.edu \\
  \textbf{website:} http://www.cive.uh.edu/faculty/nakshatrala}\\[\baselineskip]

\vspace{-0.15in}   
\begin{figure*}[h]
	\includegraphics[scale=0.2]{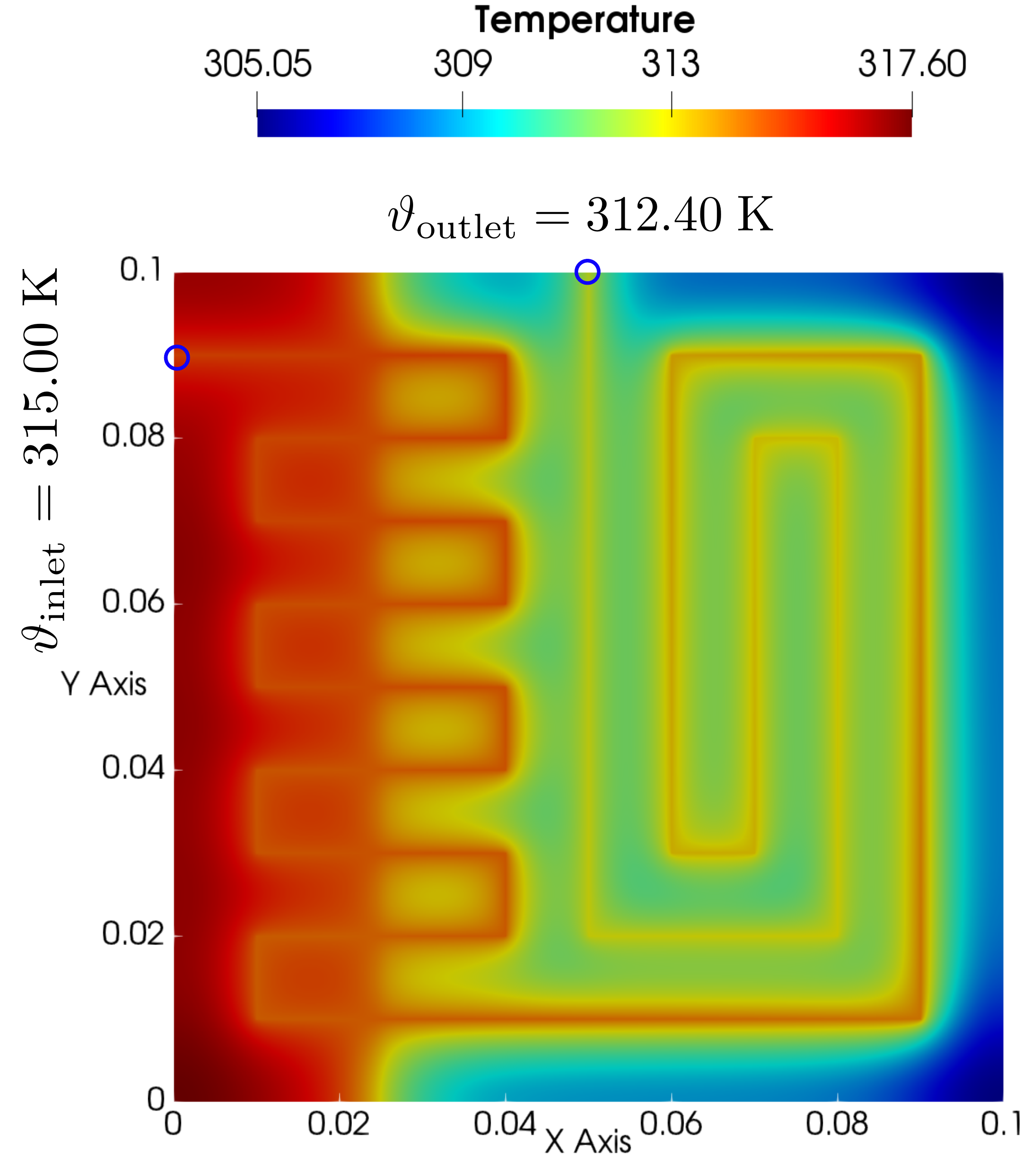}

	\emph{This figure highlights one of the ramifications of letting the inlet temperature $\vartheta_{\mathrm{inlet}}$ differ from the ambient temperature $\vartheta_{\mathrm{amb}}$. The result shows the temperature field when $\vartheta_{\mathrm{inlet}} \geq \vartheta_{\mathrm{amb}}$ and indicates that the inlet temperature is greater than the outlet temperature. On the other hand, not shown in the figure, $\vartheta_{\mathrm{inlet}} \leq \vartheta_{\mathrm{outlet}}$ when $\vartheta_{\mathrm{inlet}} =  \vartheta_{\mathrm{amb}}$.}
\end{figure*}
  \vfill
  {\scshape 2022} \\
  {\small Computational \& Applied Mechanics Laboratory} \par
\end{titlepage}

\begin{abstract}
Mimicking vascular systems in living beings, designers have realized microvascular composites to achieve thermal regulation and other functionalities, such as electromagnetic modulation, sensing, and healing. Such material systems avail circulating fluids through embedded vasculatures to accomplish the mentioned functionalities that benefit various aerospace, military, and civilian applications. Although heat transfer is a mature field, control of thermal characteristics in synthetic microvascular systems via circulating fluids is new, and a theoretical underpinning is lacking. What will benefit designers are predictive mathematical models and an in-depth qualitative understanding of vascular-based active cooling/heating. So, the central focus of this paper is to address the remarked knowledge gap. \emph{First}, we present a reduced-order model with broad applicability, allowing the inlet temperature to differ from the ambient temperature. \emph{Second}, we apply mathematical analysis tools to this reduced-order model and reveal many heat transfer properties of fluid-sequestered vascular systems. We derive point-wise properties (minimum, maximum, and comparison principles) and global properties (e.g., bounds on performance metrics such as the mean surface temperature and thermal efficiency). These newfound results deepen our understanding of active cooling/heating and propel the perfecting of thermal regulation systems. 
\end{abstract}

\maketitle

\vspace{-0.3in}

\setcounter{figure}{0}   


\section*{PRINCIPAL NOTATION}

\vspace{-0.1in}

\begin{longtable}{ll}\hline
  \multicolumn{1}{c}{\textbf{Symbol}} & \multicolumn{1}{c}{\textbf{Quantity}} \\
  \hline \multicolumn{2}{c}{\emph{Geometry-related quantities}} \\ \hline
  $\mathcal{B}$ & three-dimensional body \\ 
  $\Omega$ & domain: mid-surface of the body \\
  $\partial \Omega$ & boundary of the domain \\
  $\Gamma^{\vartheta}$ & part of the boundary with prescribed temperature \\
  $\Gamma^{q}$ & part of the boundary with prescribed heat flux \\
  $\Sigma$ & curve denoting the vasculature \\
  $d$ & thickness of the body \\ 
  $\widehat{\mathbf{n}}(\mathbf{x})$ & unit outward normal vector at $\mathbf{x}$ on the boundary \\
  $s$ & arc-length along $\Sigma$ measured from the inlet \\ 
  $\widehat{\mathbf{t}}(\mathbf{x})$ & unit tangential vector at $\mathbf{x}$ along the vasculature ($\Sigma)$ \\
  $\mathbf{x}$ & a spatial point \\
  \hline \multicolumn{2}{c}{\emph{Solution fields}} \\ \hline
  $\vartheta(\mathbf{x})$ & temperature (scalar) field \\
  $\mathbf{q}(\mathbf{x})$ & heat flux (vector) field \\
  \hline \multicolumn{2}{c}{\emph{Prescribed quantities}} \\ \hline
  $\vartheta_{\mathrm{amb}}$ & ambient temperature \\ 
  $\vartheta_{\mathrm{inlet}}$ & temperature at the inlet \\ 
  $\vartheta_{\mathrm{p}}(\mathbf{x})$ & prescribed temperature on the boundary \\ 
  $f(\mathbf{x})$ & applied heat flux \\
  $f_0$ & constant applied heat flux \\ 
  $\dot{m}$ & mass flow rate in the vasculature \\
  $q_{\mathrm{p}}(\mathbf{x})$ & prescribed heat flux on the boundary \\ 
  \hline \multicolumn{2}{c}{\emph{Material and surface properties}} \\ \hline
  $\varepsilon$ & emissivity \\
  $c_f$ & heat capacity of the fluid flowing 
  within the vasculature \\
  $h_{T}$ & convective heat transfer coefficient \\
  $\mathbf{K}(\mathbf{x})$ & thermal conductivity tensor field \\
  \hline \multicolumn{2}{c}{\emph{Other symbols}} \\
  \hline
  $\llbracket \cdot \rrbracket$ & jump operator across the vasculature $\Sigma$ \\
  $\epsilon$ & symbol introduced to define a limiting process \\
  $\eta^{e}$ & efficiency \\ 
  $\vartheta_{\mathrm{HSS}}$ & hot steady-state temperature  \\
  $\vartheta_{\mathrm{mean}}$ & mean temperature \\ 
  $\sigma$ & Stefan-Boltzmann constant $\approx 5.67\times 10^{-8} \; [\mathrm{W/m^{2}/K^4}]$ \\
  $\chi$ & heat capacity rate of the fluid, $\chi = \dot{m} \, c_f$ \\
  $\mathrm{div}[\cdot]$ & spatial divergence operator \\
  $\mathrm{grad}[\cdot]$ & spatial gradient operator \\
  \hline \multicolumn{2}{c}{\emph{Abbreviations}} \\ \hline
  %
  \textsf{HSS} & hot steady-state \\
  \textsf{PDE} & partial differential equation \\
  \hline
\end{longtable}

\vspace{-0.1in}

\section{INTRODUCTION AND MOTIVATION}
\label{Sec:S1_ROM_Qualitative_Intro}

\lettrine[findent=2pt]{\fbox{\textbf{T}}}{hermal regulation} is vital to homeostasis, which is essential for the survival of living organisms. Homeostasis, a biology term, refers to maintaining controlled physical and chemical conditions within the body and organs \citep{reece2014campbell}. One of the mechanisms by which animals control their internal temperature is via the circulation of blood through a pervasive vasculature comprising arteries, veins, and capillaries. Such fluid circulation can facilitate either \emph{active cooling} or \emph{active heating}. For example, jackrabbits adapt to adverse desert climate via active cooling: they use blood flow in their large ears exposed to the surroundings to lower their body temperature \citep{hill1976jackrabbit}. On the other hand, seagulls handle ``cold feet" via active heating: the flow of warm blood in arteries and veins maintains core body temperature while standing in the snow \citep{scholander1957countercurrent,livingbirdmagazine}. Humans, too, maintain a healthy internal temperature via blood circulation through a network of vessels \citep{hall2020guyton}.

Even in the synthetic world, thermal regulation is essential to the success of many technological endeavors: high-powered antennas \citep{lyall2008experimental}, microchip electronics \citep{dede2018thermal}, hypersonic space vehicles \citep{gou2019design}, battery packing \citep{pety2017active}, to name a few. By embracing biomimicry, researchers have realized thermal regulation in structural systems by embedding a network of vesicles through which fluids are circulated \citep{kozola2010characterization,rocha2009tree}. With such experimental platforms at their disposal, engineers are looking to achieve optimal designs and gain a deeper understanding of the underlying physics. But thermal regulation in fluid-sequestered vascular systems is a complicated phenomenon. Geometrical attributes (vasculature layout, spacing), material properties (conductivity of the host material, surface heat transfer characteristics, and heat capacity and density of the fluid), and prescribed inputs (inlet temperature, ambient temperature, heat source, boundary conditions, flow rate within the vasculature) affect the performance of a thermal regulating system. 

Prior works to unravel thermal regulation in synthetic materials can be classified into three broad categories. Under the \emph{first} category, researchers focused on fabrication techniques for creating a network of vesicles in fiber-reinforced composites (e.g., the VaSC---vaporization of sacrificial components---technique) \citep{esser2011three}. The \emph{second} set of studies focused on developing numerical formulations and demonstrating the capabilities of these formulations in handling complex vasculatures \citep{tan2015nurbs}. The \emph{third} set of studies employed optimization techniques, such as shape and topology optimization, to get optimal vasculature layouts \citep{mcelroy2015optimisation,tan2016gradient,pejman2019gradient}. 

These tools---fabrication techniques, experimental platforms, and computational algorithms---are essential contributions. But given the complexity and myriad of variables affecting the phenomenon, what advances the field further are general qualitative properties that are not specific to a boundary value problem or an experimental setup. On that account, this paper (a) develops such qualitative properties using mathematical analysis and (b) studies the ramifications of these properties on thermal management in vascular systems. 

To undertake such an analysis, we need an appropriate mathematical model. A nonlinear coupled model involving energy-momentum equations---the so-called conjugate heat transfer---will completely resolve all the underlying processes \citep{perelman1961conjugated,dorfman2009conjugate}. A mathematical analysis of such a model is daunting because of the nonlinear and coupled nature of the governing equations. However, when the vasculature is thin, and the flow within is laminar, one can simplify the analysis. In the literature, a reduced-order model, approximating the heat transfer across the fluid and host solid, has been widely used, primarily for obtaining numerical solutions; for instance, the works by \citet{aragon2010generalized} and \citet{tan2015nurbs}. This paper will use such a reduced-order model in the mathematical analysis, but with some generalizations. 

Previous modeling efforts, and even experimental studies, considered the fluid's temperature at the inlet to be the same as the ambient temperature \citep{pejman2019gradient,devi2021microvascular}. However, such a requirement is stringent and limits the possibilities, especially for thermal regulation in harsh environments. Also, from a practical point of view, it is not always possible to have the fluid's temperature be the same as ambient. Specifically, ambient temperatures can be extreme in space applications. For instance, in deep space and planetary explorations, a spacecraft can experience temperatures that are a few dozen Kelvins over the absolute zero \citep{gilmore2002spacecraft}. In contrast, temperature at the belly of a reentry vehicle can exceed 2000 K \citep{finke1990calculation}. Thus, letting the inlet and ambient temperatures differ provides more flexibility in regulating the temperature field. Ergo, this paper presents a mathematical model with the said flexibility and discusses the associated ramifications. The resulting model furnishes a boundary value problem involving a second-order nonlinear elliptic partial differential operator with a unique jump condition arising from the energy balance, involving tangential derivative along the vasculature. 

Second-order elliptic partial differential equations (PDEs) satisfy unique properties such as maximum principles, comparison principles, and monotone properties \citep{pao2012nonlinear}. These principles provide general qualitative properties on the nature of the solutions without actually solving the boundary value problems; such an approach is particularly fruitful in studying nonlinear PDEs, as analytical solutions are scarce \citep{gilbarg2015elliptic}.
Such principles also find immense use in studying engineering problems and developing accurate numerical formulations. For example, in the theory of torsion, the occurrence of the maximum shear stress on the boundary can be shown to be a consequence of a maximum principle \citep{love2013treatise,sperb1981maximum,selvadurai2000partial}. A topic of great interest in computational mechanics is developing numerical formulations that preserve these principles. For instance, Nakshatrala and co-workers have developed numerical formulations for diffusion-type equations that satisfy maximum principles \citep{nakshatrala2013numerical,mudunuru2016enforcing,karimi2016current,nakshatrala2016numerical}. Holding such fundamental mathematical principles in the discrete setting is often considered essential for deeming a numerical formulation predictive.

Motivated by their utility and power, this paper derives several such mathematical principles for the chosen reduced-order model. However, due to different \emph{regularity} (smoothness) requirements and the presence of a unique jump condition (along the vasculature) in the reduced-order model, the prior results presented in the literature for second-order PDEs are not directly applicable. Hence, the derivations presented in this paper use alternative strategies. These newly obtained principles describe the general nature of the temperature field without actually solving a boundary value problem and establish the reduced-order model to be well-posed. Given the newfound understanding from the mathematical analysis, this paper reexamines the question: \emph{What is an appropriate definition to assess the efficiency of thermal regulation}? We reveal the deficiencies of the currently used definition of efficiency. Consequently, we provide alternative notions of efficiency for thermal regulation, considering both active cooling and active heating. 

An outline for the rest of this article is as follows. We first provide the governing equations of a reduced-order model for thermal regulation in fluid-sequestered thin vascular systems, allowing different inlet and ambient temperatures (\S\ref{Sec:S2_ROM_Qualitative_GE}). Next, several mathematical properties are derived, neglecting radiation in the reduced-order model (\S\ref{Sec:S3_ROM_Qualitative_Mathematical_analysis}). Using these results, we then get the corresponding properties and ramifications when radiation is included (\S\ref{Sec:S4_ROM_Qualitative_Radiation}). To put the newly obtained results in the light of prior experiments, we consider a particular case: constant heat source and adiabatic lateral boundaries (\S\ref{Sec:S5_ROM_Qualitative_Special_case}). A ramification of inlet and ambient temperatures differing, which addresses a fundamental question of thermal regulation, is then presented (\S\ref{Sec:S6_ROM_TOutlet_LT_TInlet}). After that, the definition of efficiency for thermal regulation is re-examined and alternatives are proposed (\S\ref{Sec:S7_ROM_Qualitative_Efficiency}). The paper closes with concluding remarks, including a brief outline of some relevant future work on thermal regulation (\S\ref{Sec:S8_ROM_Qualitative_Closure}).

\section{PROBLEM DESCRIPTION AND A REDUCED-ORDER MODEL}
\label{Sec:S2_ROM_Qualitative_GE}
Consider a thin body $\mathcal{B}$ in the ambient space $\mathbb{R}^{3}$ with thickness $d$, which is assumed to be much smaller than the other characteristic dimensions of the body. The body is assumed to be flat (i.e., no curvature) and bounded. An external source heats the body. A fluid is pushed to flow through an embedded vascular network---regulating the body’s temperature. \textbf{Figure~\ref{Fig:1_ROM_Qualitative_Problem_description}} provides a pictorial description of this setup. We now present a mathematical description of the mentioned thermal regulation in vascular systems. 

Since the body is thin, we can adequately describe the concomitant heat transfer using a reduced-order model codified in two spatial dimensions, instead of appealing to a complete three dimensional treatment. On this account, we denote the two-dimensional embedding of the mid-surface of this body by $\Omega \subset \mathbb{R}^{2}$; $\Omega$ will be referred to as the \emph{domain} in the rest of this paper. As done in the theory of partial differential equations, we assume $\Omega$ to be an open set. The boundary of $\Omega$ is denoted by $\partial \Omega := \overline{\Omega} - \Omega$, where a superposed bar represents the set closure. A spatial point is denoted by $\mathbf{x} \in \overline{\Omega}$. The outward unit normal vector to the boundary is denoted by $\widehat{\mathbf{n}}(\mathbf{x})$. The spatial divergence and gradient operators are denoted by $\mathrm{div}[\cdot]$ and $\mathrm{grad}[\cdot]$, respectively.

\begin{figure}[ht]
    \centering
    \includegraphics[scale=0.6]{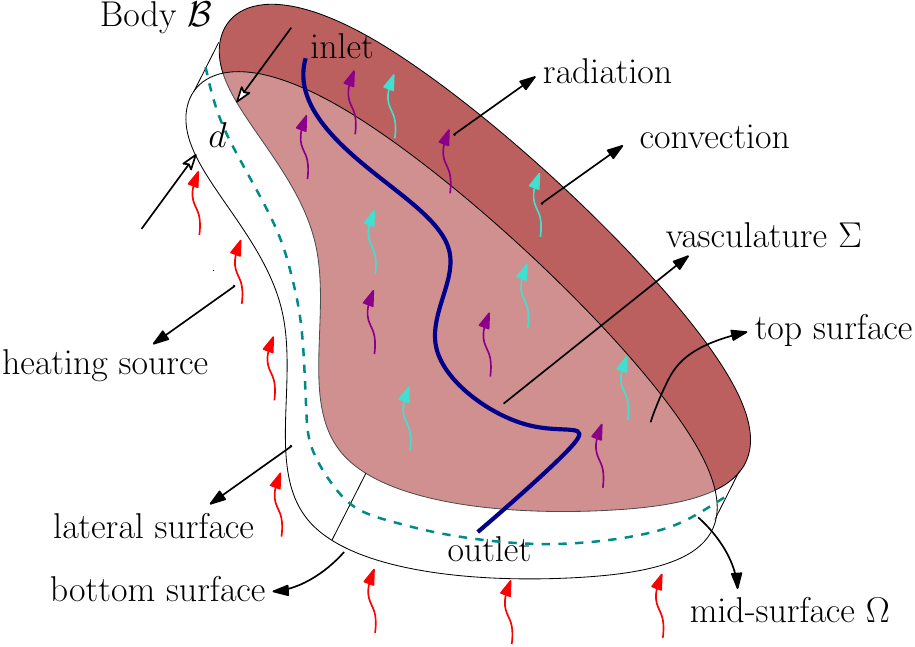}
    \caption{This figure illustrates a vascular-based thermal regulation setup. The three-dimensional body with thickness $d$ is denoted by $\mathcal{B}$, the mid-surface is by $\Omega$, and the vasculature is by $\Sigma$. 
    \label{Fig:1_ROM_Qualitative_Problem_description}}
\end{figure}

The vasculature network is modeled as a \emph{curve} within $\Omega$ and is denoted as $\Sigma$; see \textbf{Fig.~\ref{Fig:2_ROM_Qualitative_Singular_surface}}. This curve's starting point (inlet) and ending point (outlet) lie on the boundary $\partial\Omega$ with the arc-length parameter increasing from the inlet to the outlet. A unit tangent vector in the direction of increasing arc-length at a spatial point on this curve is denoted by $\widehat{\mathbf{t}}(\mathbf{x})$. Since the vasculature is fixed in space and contains a flowing fluid, it can be considered as a non-moving energetic singular surface with its own rate of energy production. The balance laws across such a surface are governed by jump conditions.

\begin{figure}
    \centering
    \includegraphics[scale=0.9]{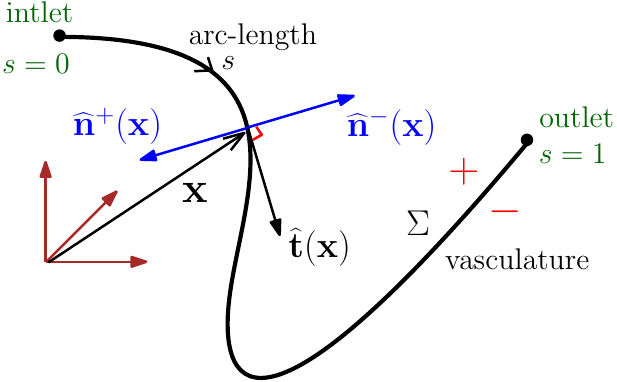}
    \caption{The vasculature $\Sigma$ divides the domain into two subdomains, denoted by $-$ and $+$. The concomitant outward unit normals for these subdomains at a spatial point $\mathbf{x}$ are denoted by $\widehat{\mathbf{n}}^{-}(\mathbf{x})$ and $\widehat{\mathbf{n}}^{+}(\mathbf{x})$. The arc-length, denoted by $s$, is used to parameterize the curve $\Sigma$ with $s = 0$ denoting the inlet while $s = 1$ the outlet. The unit tangent vector at $\mathbf{x}$ along increasing $s$ is denoted by $\widehat{\mathbf{t}}(\mathbf{x})$. \label{Fig:2_ROM_Qualitative_Singular_surface}}
\end{figure}

To this end, we define the limiting values of the fields on the either side of $\Sigma$ and the jump operator. The vasculature divides the domain $\Omega$ into two parts: denoted by $-$ and $+$, as indicated in Fig.~\ref{Fig:2_ROM_Qualitative_Singular_surface}. The limiting values of the temperature field on the either side of $\Sigma$ are defined as follows: 
\begin{align}
    \label{Eqn:ROM_Qualitative_Limiting_values}
    \vartheta^{-}(\mathbf{x}) = \lim_{\epsilon \rightarrow 0^{+}} 
    \, \vartheta\big(\mathbf{x} - \epsilon \,   \widehat{\mathbf{n}}^{-}(\mathbf{x})\big) 
    \quad \mathrm{and} \quad 
    \vartheta^{+}(\mathbf{x}) = \lim_{\epsilon \rightarrow 0^{+}} 
    \, \vartheta\big(\mathbf{x} - \epsilon \,   \widehat{\mathbf{n}}^{+}(\mathbf{x})\big)
\end{align}
where $\epsilon \rightarrow 0^{+}$ indicates the limit approaching from the positive side of the real line. Likewise, one can define the limiting values $\mathbf{q}^{\pm}(\mathbf{x})$ for the heat flux vector. The jump in fields across $\Sigma$ is denoted by the jump operator: 
\begin{subequations} 
    \label{Eqn:ROM_Qualitative_jump_operator} 
    \begin{align}
        \label{Eqn:ROM_Qualitative_jump_operator_scalar} 
        &\llbracket\vartheta(\mathbf{x}) \rrbracket 
        = \vartheta^{+}(\mathbf{x})
        \widehat{\mathbf{n}}^{+}(\mathbf{x}) 
        + \vartheta^{-}(\mathbf{x}) 
        \widehat{\mathbf{n}}^{-}(\mathbf{x}) \\
        \label{Eqn:ROM_Qualitative_jump_operator_vector} 
        &\llbracket\mathbf{q}(\mathbf{x}) \rrbracket = \mathbf{q}^{+}(\mathbf{x})
        \cdot \widehat{\mathbf{n}}^{+}(\mathbf{x}) 
        + \mathbf{q}^{-}(\mathbf{x}) 
        \cdot \widehat{\mathbf{n}}^{-}(\mathbf{x}) 
    \end{align}
\end{subequations}
Note that the jump operator defined above acts on a scalar field to produce a vector field and \emph{vice versa}.

One of the jump conditions is the continuity of the temperature across the vasculature. Stated mathematically, $\llbracket \vartheta(\mathbf{x})\rrbracket = \mathbf{0}$ on $\Sigma$.
The second jump condition is the balance of energy across $\Sigma$. To derive the corresponding mathematical expression, we consider a differential segment of the vasculature, as shown in \textbf{Fig.~\ref{Fig:ROM_Energy_jump_condition}}.  As mentioned earlier, a fluid (referred to as a coolant in the case of active cooling) flows through the vascular network with a mass flow rate $\dot{m}$ and coefficient of heat capacity $c_f$. The rate of thermal energy carried by the fluid into the vasculature at the inlet is $\chi \vartheta$, in which $\chi$ is the heat capacity rate, defined as: 
\begin{align}
    \label{Eqn:ROM_Qualitative_heat_capacity_rate}
    \chi = \dot{m} \, c_f 
\end{align}
The rate of energy carried by the fluid out of the vasculature at the outlet is $\chi (\vartheta + \frac{\mathrm{d}\vartheta}{\mathrm{d}s}\mathrm{d}s)$. The heat fluxes into the vasculature from these two regions are, respectively, indicated as $\mathbf{q}^{+}(\mathbf{x}) \cdot \widehat{\mathbf{n}}^{+}(\mathbf{x})$ and $\mathbf{q}^{-}(\mathbf{x}) \cdot \widehat{\mathbf{n}}^{-}(\mathbf{x})$. An energy balance on this differential segment of the vasculature implies that 
\begin{align}
\chi \vartheta + \left(\mathbf{q}^{+}(\mathbf{x}) \cdot \widehat{\mathbf{n}}^{+}(\mathbf{x}) + \mathbf{q}^{-}(\mathbf{x}) \cdot \widehat{\mathbf{n}}^{-}(\mathbf{x})\right) \mathrm{d}s = \chi \left(\vartheta + \frac{\mathrm{d}\vartheta}{\mathrm{d}s} \, \mathrm{d}s\right)
\end{align}
Hence, the jump condition for energy balance across $\Sigma$ can be compactly written as $\llbracket \mathbf{q}\rrbracket = \chi \frac{\mathrm{d}\vartheta}{\mathrm{d}s} = \chi \mathrm{grad}[\vartheta]\cdot \widehat{\mathbf{t}}(\mathbf{x})$.

\begin{figure}[h]
    \centering
    \includegraphics[scale=0.65,clip]{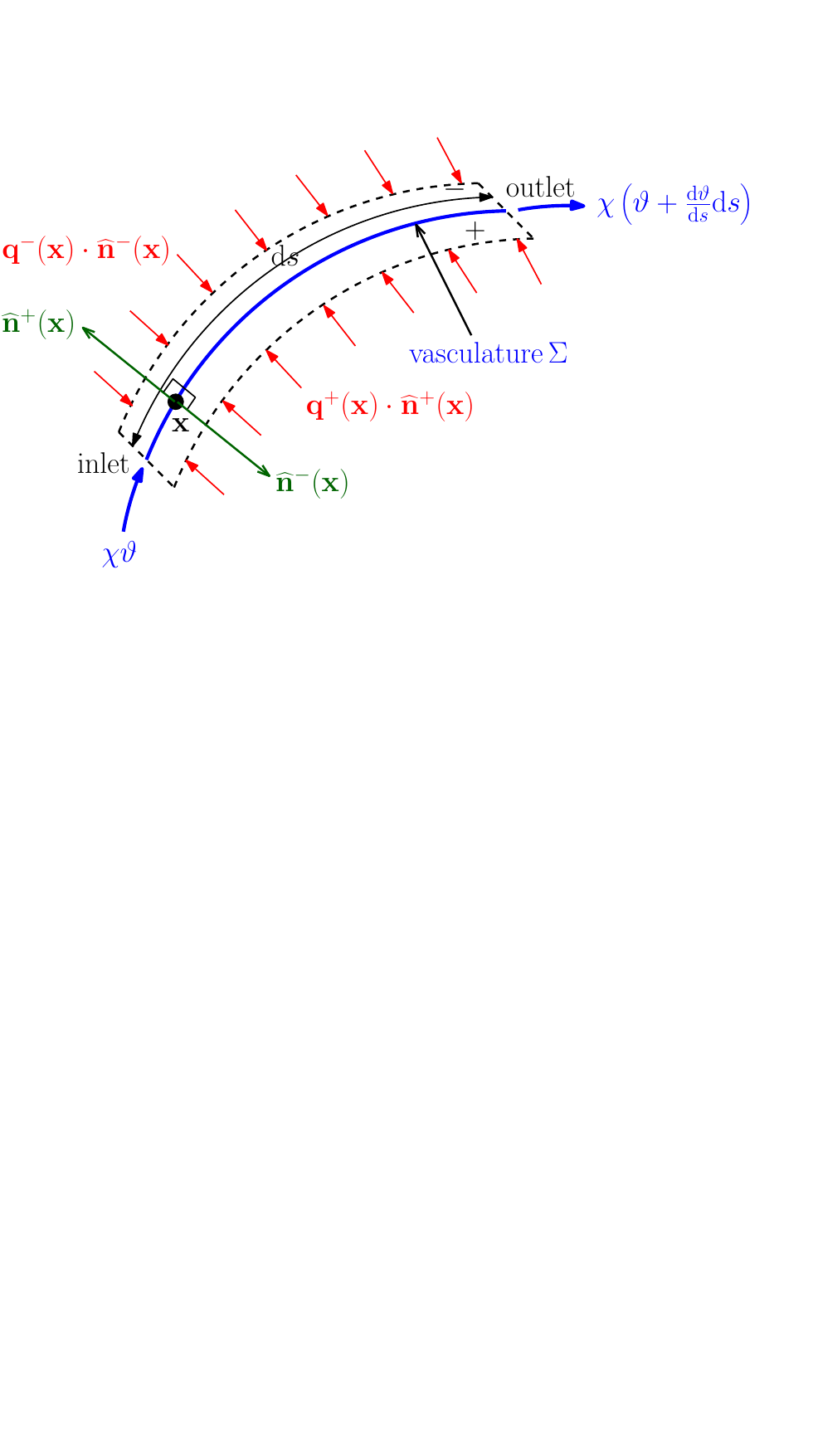}
    \caption{The figure show the energy balance for a differential segment $\mathrm{d}s$ of the vasculature $\Sigma$. The regions on the either side of the vasculature are denoted by ``$+$" and ``$-$". \label{Fig:ROM_Energy_jump_condition}}
\end{figure}


In addition to the transport of energy due to the flow of a fluid in the vasculature, we consider three other modes of heat transfer: heat conduction within the bulk solid, convection, and radiation. Newton's model for cooling describes the convective heat transfer while the Stefan-Boltzmann model the radiation. The Newton's model for cooling describes the convective heat transfer while the Stefan-Boltzmann model the radiation. We use the Fourier model for the heat conduction within the bulk: 
\begin{align}
    \label{Eqn:ROM_Qualitative_Fourier_model}
    \mathbf{q}(\mathbf{x}) = - \mathbf{K}(\mathbf{x}) \, \mathrm{grad}[\vartheta(\mathbf{x})]
\end{align}
where $\mathbf{K}(\mathbf{x})$ is the coefficient of thermal conductivity: a symmetric tensor field---possibly anisotropic and spatially inhomogeneous. The governing equations for the reduced-order model take the following form: 
\begin{subequations}
\begin{alignat}{2}
    \label{Eqn:ROM_Qualitative_BoE_general} 
    &d \, \mathrm{div}[\mathbf{q}(\mathbf{x})] = f(\mathbf{x}) 
    - h_{T} \, (\vartheta(\mathbf{x}) - \vartheta_{\mathrm{amb}}) 
    - \varepsilon \, \sigma \,  (\vartheta^{4}(\mathbf{x}) - \vartheta^{4}_{\mathrm{amb}}) 
    && \quad \mathrm{in} \; \Omega \setminus \Sigma \\
    \label{Eqn:ROM_Qualitative_q_jump_condition_general} 
    &\llbracket\mathbf{q}(\mathbf{x})\rrbracket = \chi \, \mathrm{grad}[\vartheta] \cdot \widehat{\mathbf{t}}(\mathbf{x}) 
    && \quad \mathrm{on} \; \Sigma \\
    \label{Eqn:ROM_Qualitative_temp_jump_condition_general} 
    &\llbracket\vartheta(\mathbf{x})\rrbracket = \boldsymbol{0}
    && \quad \mathrm{on} \; \Sigma \\
     \label{Eqn:ROM_Qualitative_q_BC_general} 
    &d \, \mathbf{q}(\mathbf{x}) \cdot \widehat{\mathbf{n}}(\mathbf{x}) 
    = q_{\mathrm{p}}(\mathbf{x}) 
    && \quad \mathrm{on} \; \Gamma^{q}  \\
    \label{Eqn:ROM_Qualitative_temp_BC_general} 
    &\vartheta(\mathbf{x}) = \vartheta_{\mathrm{p}}(\mathbf{x}) 
    && \quad \mathrm{on} \; \Gamma^{\vartheta} \\
    \label{Eqn:ROM_Qualitative_inlet_general} 
    &\vartheta = \vartheta_{\mathrm{inlet}}
    && \quad s = 0 \; \mathrm{on} \; \Sigma  
\end{alignat}
\end{subequations}
where $\vartheta_{\mathrm{p}}(\mathbf{x})$ is the prescribed temperature on the boundary, $q_{\mathrm{p}}(\mathbf{x})$ is the prescribed heat flux on the boundary, $h_T$ is the convective heat transfer coefficient, $\sigma \approx 5.67 \times 10^{-8} \; \mathrm{W/m^2/K^4}$ is the Stefan-Boltzmann constant, and $\varepsilon$ is the emissivity\footnote{Note the notational difference between $\epsilon$, denoting the limiting process used in Eq.~\eqref{Eqn:ROM_Qualitative_Limiting_values}, and $\varepsilon$---denoting the emissivity.}. Note that the inlet temperature in Eq.~\eqref{Eqn:ROM_Qualitative_inlet_general} need not be the same as the ambient temperature.

The thermal conductivity tensor is assumed to be uniformly bounded. That is, there exist two constants $0 < \kappa_1 \leq \kappa_2$ such that 
\begin{align}
    0 < 
    \kappa_{1} \, \boldsymbol{\xi}(\mathbf{x}) 
    \cdot \boldsymbol{\xi}(\mathbf{x}) 
    \leq \boldsymbol{\xi}(\mathbf{x}) \cdot 
    \mathbf{K}(\mathbf{x}) \boldsymbol{\xi}(\mathbf{x}) 
    \leq \kappa_{2} \, \boldsymbol{\xi}(\mathbf{x}) 
    \cdot \boldsymbol{\xi}(\mathbf{x}) 
    \quad \forall \mathbf{x} \in \Omega, 
    \forall \boldsymbol{\xi}(\mathbf{x}) \neq 
    \boldsymbol{0} 
\end{align}
Requiring a tensor field to be uniformly bounded below, commonly used in mathematical analysis and needed for technicality, is more stringent than the tensor field being positive definite \citep{gilbarg2015elliptic}. The boundary value problem \eqref{Eqn:ROM_Qualitative_BoE_general}--\eqref{Eqn:ROM_Qualitative_inlet_general} with radiation is a semi-linear---also referred to as mildly nonlinear---elliptic partial differential equation (\textsf{PDE}). Without radiation, the boundary value problem is a linear elliptic \textsf{PDE}. 

Many mathematical studies---deriving estimates and qualitative properties such as maximum and comparison principles---exist in the literature for a hierarchy of models \emph{without} active cooling (i.e., $\chi = 0$); some related authoritative references include \citep{gilbarg2015elliptic,mcowen1996partial}. However, such investigations are \emph{unavailable} for the active-cooling model given by Eqs.~\eqref{Eqn:ROM_Qualitative_BoE_general}--\eqref{Eqn:ROM_Qualitative_inlet_general}. Therefore, in subsequent sections, we conduct a systematic mathematical analysis of the considered active-cooling model and show that this model is well-posed and predictive for thermal regulation in vascular systems. \textsc{Manfred Eigen}---a Nobel Prize winner in Chemistry---echoes the necessity of such a study:
\begin{quote}
``\emph{A theory has only the alternative of being right or wrong. 
A model has a third possibility: it may be right, but irrelevant.}"
\end{quote} 

\section{MATHEMATICAL ANALYSIS}
\label{Sec:S3_ROM_Qualitative_Mathematical_analysis}

The plan for this section is to  establish results neglecting radiation. We will first state the Galerkin weak formulation corresponding to the strong form of the reduced-order model. We then utilize this formulation in establishing minimum and maximum principles. Appealing to these two principles, we prove a comparison principle, the uniqueness of solutions, and continuous dependence of solutions on the input---also known as the stability of solutions. We will provide theorem statements in plain language to convey the import of intricate theorems, thereby appealing to a broader audience. For involved mathematical proofs, we will give a small discussion---called ``Proof Idea"---outlining the employed strategy: this is in addition to providing rigorous proofs. 

To state the weak formulation precisely, we define the following function spaces:
\begin{subequations}
\begin{align}
    \mathcal{U} &:= 
    \left\{ \vartheta(\mathbf{x}) \in H^{1}(\Omega) \; \vert \; \vartheta(\mathbf{x}) = \vartheta_{\mathrm{p}}(\mathbf{x}) \; \mathrm{on} \; \Gamma^{\vartheta} \; \mathrm{and} \; \vartheta(\mathbf{x}) = \vartheta_{\mathrm{inlet}} \; \mathrm{at} \; s = 0 \; \mathrm{on} \; \Sigma \right\} \\ 
    \mathcal{W} &:= 
    \left\{ w(\mathbf{x}) \in H^{1}(\Omega) \; \vert \; 
    w(\mathbf{x}) = 0 \; \mathrm{on} \; 
    \Gamma^{\vartheta} \; \mathrm{and} \; w(\mathbf{x}) = 0 \; \mathrm{at} \; s = 0 \; \mathrm{on} \; \Sigma \right\}
\end{align}
\end{subequations}
in which $H^{1}(\Omega)$ is a standard Sobolev space containing all the square-integrable functions alongside their first weak derivatives are also square-integrable \citep{evans1998partial}. 
We use $L_2(\mathcal{K})$ to denote the set of square integrable functions defined over $\mathcal{K}$. In these remarked function spaces, the integration theory is the Lebesgue integration \citep{bartle2014elements}. An associated concept is \emph{almost everywhere}, which is commonly abbreviated as ``a.e." in mathematical analysis. Two functions are equal almost everywhere if these functions differ only on a set of measure zero; mathematically, the said  statement is written as: $f(\mathbf{x}) = g(\mathbf{x}) \;  \mathrm{a.e.}$. We denote the set of continuous functions over $\mathcal{K}$ by $C^{0}(\mathcal{K})$, and the set of once continuously differentiable functions over a set $\mathcal{K}$ by $C^{1}(\mathcal{K})$.

The \emph{Galerkin weak  formulation} corresponding to the boundary value problem \eqref{Eqn:ROM_Qualitative_BoE_general}--\eqref{Eqn:ROM_Qualitative_inlet_general} neglecting radiation reads: Find $\vartheta(\mathbf{x}) \in \mathcal{U}$ such that we have
\begin{align}
    \label{Eqn:ROM_Qualitative_Galerkin_weak_form}
    \int_{\Omega} d \, \mathrm{grad}[w] \cdot \mathbf{K}(\mathbf{x}) 
    \mathrm{grad}[\vartheta] \, \mathrm{d} \Omega 
    &+ \int_{\Omega} h_{T} \, w(\mathbf{x}) (\vartheta(\mathbf{x}) - 
    \vartheta_{\mathrm{amb}}) \, \mathrm{d} \Omega 
    \nonumber \\
    &+ \int_{\Sigma} \chi \, w(\mathbf{x}) \mathrm{grad}[\vartheta] \cdot \widehat{\mathbf{t}}(\mathbf{x}) \, \mathrm{d} \Gamma \nonumber \\
    &= \int_{\Omega} w(\mathbf{x}) f(\mathbf{x}) \, \mathrm{d} \Omega 
    - \int_{\Gamma^{q}} w(\mathbf{x}) q_{\mathrm{p}}(\mathbf{x}) \, \mathrm{d} \Gamma 
    \quad \forall w(\mathbf{x}) \in \mathcal{W} 
\end{align}
The finite element method often uses the Galerkin weak formulation to get numerical solutions. Herein, we use this formulation to establish the following significant result of this paper: a minimum principle that a solution field of the reduced-order model satisfies. 

\begin{theorem}[A minimum principle]
    \label{Thm:ROM_Qualitative_Minimum_principle}
    Let $f(\mathbf{x}) \in L_2(\Omega)$ and $q_{\mathrm{p}}(\mathbf{x}) \in L_2(\Gamma^{q})$, and 
    $\vartheta(\mathbf{x}) \in C^{1}(\Omega\setminus\Sigma) \cap C^{0}(\overline{\Omega})$ be a solution of the Galerkin weak formulation. If 
    \[
    f(\mathbf{x}) \geq 0 \quad \mathrm{a.e.} 
    \quad \mathrm{and} \quad 
    q_{\mathrm{p}}(\mathbf{x}) \leq 0
    \quad \mathrm{a.e.} 
    \]
    then the solution field $\vartheta(\mathbf{x})$ satisfies the following lower bound: 
    \begin{align}
    \label{Eqn:ROM_Qualitative_Minimum_principle_statement}
    \min\Big[\vartheta_{\mathrm{amb}},\vartheta_{\mathrm{inlet}}, \min_{\mathrm{x} \in \Gamma^{\vartheta}}\big[ \vartheta_{\mathrm{p}}(\mathbf{x})\big]\Big] \leq \vartheta(\mathbf{x}) 
    \quad \forall \mathbf{x} \in \overline{\Omega}
    \end{align}
\end{theorem}
\begin{layperson}
Consider that the applied heat within the domain is a source---heat is supplied to the system. The heat flux prescribed on a portion or entire boundary is a sink (i.e., heat is taken out of the system), while the rest of the boundary is prescribed with a  temperature. Then, the temperature everywhere in the domain will \emph{not} be lower than the minimum amongst the ambient temperature, the temperature at the inlet, and the minimum of the prescribed temperature on the boundary.
\end{layperson}
\begin{proofidea}
The crux of the proof is to construct a convenient norm---a non-negative functional\footnote{In this paper, it will suffice to define a \emph{functional} as a function of a (input) function. For a more detailed mathematical exposure on functionals, see \citep{gelfand2000calculus}.}---and use the fact that a norm is zero if and only if the function defining it itself is zero. Our approach for constructing this norm hinges on using the Galerkin weak formulation (cf. Eq.~\eqref{Eqn:ROM_Qualitative_Galerkin_weak_form}) and defining an auxiliary function (cf. Eq.~ \eqref{Eqn:ROM_Qualitative_definition_eta}). This auxiliary function, defined in terms of temperature field, is constructed so that it is \emph{zero} whenever the statement of the minimum principle is true and \emph{negative} otherwise (cf. Eqs.~\eqref{Eqn:ROM_Qualitative_definition_Phi_min} and \eqref{Eqn:ROM_Qualitative_definition_eta}). We then rewrite the Galerkin weak form in terms of this newly introduced auxiliary function (cf.~Eq.~\eqref{Eqn:ROM_Qualitative_Minimum_principle_eta_substitution}). Using the properties of the auxiliary function, we simplify the resulting expression to get a norm (in terms of the auxiliary function) and establish that this norm is equal to zero (cf.~Eq.~\eqref{Eqn:ROM_Qualitative_Min_Principle_derivation_Step_2}). By definition, vanishing norm implies that the auxiliary function itself is zero (cf.~Eq.~\eqref{Eqn:ROM_Qualitative_Min_Principle_derivation_Step_3}). By construction, the vanishing auxiliary function means that the theorem's statement is true (cf.~Eq.~\eqref{Eqn:ROM_Qualitative_Min_Principle_derivation_Step_5}). 
\end{proofidea}
\begin{proof}
Let 
\begin{align}
    \label{Eqn:ROM_Qualitative_definition_Phi_min}
    \Phi_{\min} := 
    \min\Big[\vartheta_{\mathrm{amb}},\vartheta_{\mathrm{inlet}}, \min_{\mathrm{x} \in \Gamma^{\vartheta}}\big[ \vartheta_{\mathrm{p}}(\mathbf{x})\big]\Big] 
\end{align}
We then define
\begin{align}
    \label{Eqn:ROM_Qualitative_definition_eta}
    \eta(\mathbf{x}) := \min\big[0,\vartheta(\mathbf{x}) - \Phi_{\min}\big]
\end{align}
This newly introduced field variable satisfies the following properties: 
\begin{enumerate}[(i)]
    \item $\eta(\mathbf{x}) \leq 0 \; \forall \mathbf{x} \in \overline{\Omega}$
    \item $\eta(\mathbf{x}) = 0$ on $\forall \mathbf{x} \in \Gamma^{\vartheta}$ 
    \item $\eta(\mathbf{x}) = 0$ at the inlet (i.e., $s = 0$)
    \item $\eta(\mathbf{x}) \in C^{1}(\Omega \setminus \Sigma) \cap C^{0}(\overline{\Omega})$
    \item at a given $\mathbf{x}$, either $\eta(\mathbf{x}) = 0$ 
    or $\vartheta(\mathbf{x}) = \eta(\mathbf{x}) + \Phi_{\min}$
\end{enumerate}

Due to property (ii), $\eta(\mathbf{x})$ can serve as a weighting function. Thus, by choosing $w(\mathbf{x}) = \eta(\mathbf{x})$ in the Galerkin weak formulation \eqref{Eqn:ROM_Qualitative_Galerkin_weak_form}, we get:
{\small
\begin{align}
    \label{Eqn:ROM_Qualitative_Minimum_principle_eta_substitution}
    \int_{\Omega} d \, \mathrm{grad}[\eta] \cdot \mathbf{K}(\mathbf{x}) 
    \mathrm{grad}[\vartheta] \, \mathrm{d} \Omega 
    + \int_{\Omega} h_{T} \, \eta(\mathbf{x}) (\vartheta(\mathbf{x}) - 
    \vartheta_{\mathrm{amb}}) \, \mathrm{d} \Omega 
    &+ \int_{\Sigma} \chi \, \eta(\mathbf{x}) \mathrm{grad}[\vartheta] \cdot \widehat{\mathbf{t}}(\mathbf{x}) \, \mathrm{d} \Gamma \nonumber \\
    &= \int_{\Omega} \eta(\mathbf{x}) f(\mathbf{x}) \, \mathrm{d} \Omega 
    - \int_{\Gamma^{q}} \eta(\mathbf{x}) q_{\mathrm{p}}(\mathbf{x}) \, \mathrm{d} \Gamma 
\end{align}}
In the above equation, all the terms that contain $\vartheta(\mathbf{x})$ also have $\eta(\mathbf{x})$ (i.e., the first three terms). By virtue of property (v) --- if $\vartheta(\mathbf{x})$ is not equal to $\eta(\mathbf{x}) + \Phi_{\min}$ then $\eta(\mathbf{x}) = 0$ --- we can replace $\vartheta(\mathbf{x})$ with $\eta(\mathbf{x}) + \Phi_{\min}$ in the above equation. Accordingly, we get: 
{\small
\begin{align}
    \int_{\Omega} d \, \mathrm{grad}[\eta] \cdot \mathbf{K}(\mathbf{x}) 
    \mathrm{grad}[\eta] \, \mathrm{d} \Omega 
    &+ \int_{\Omega} h_{T} \, \eta(\mathbf{x}) (\eta(\mathbf{x}) + \Phi_{\min} - 
    \vartheta_{\mathrm{amb}}) \, \mathrm{d} \Omega \nonumber \\
    &+ \int_{\Sigma} \chi \, \eta(\mathbf{x}) \mathrm{grad}[\eta] \cdot \widehat{\mathbf{t}}(\mathbf{x}) \, \mathrm{d} \Gamma 
    = \int_{\Omega} \eta(\mathbf{x}) f(\mathbf{x}) \, \mathrm{d} \Omega 
    - \int_{\Gamma^{q}} \eta(\mathbf{x}) q_{\mathrm{p}}(\mathbf{x}) \, \mathrm{d} \Gamma 
\end{align}}
In writing the above equation, we have used the fact that $\Phi_{\min}$ is independent of $\mathbf{x}$. 

Noting that $\eta(\mathbf{x}) \leq 0$ (i.e., property (i)), and  $f(\mathbf{x}) \geq 0$ and 
$q_{\mathrm{p}}(\mathbf{x}) \leq 0$ a.e., and rearranging the terms, we get: 
{\small
\begin{align}
    \int_{\Omega} d \, \mathrm{grad}[\eta] \cdot \mathbf{K}(\mathbf{x}) 
    \mathrm{grad}[\eta] \, \mathrm{d} \Omega 
    + \int_{\Omega} h_{T} \, \eta(\mathbf{x}) (\eta(\mathbf{x}) + \Phi_{\min} - 
    \vartheta_{\mathrm{amb}}) \, \mathrm{d} \Omega 
    + \int_{\Sigma} \chi \, \eta(\mathbf{x}) \mathrm{grad}[\eta] \cdot \widehat{\mathbf{t}}(\mathbf{x}) \, \mathrm{d} \Gamma 
    \leq 0
\end{align}}
Integrating the last integral, noting that $\eta(\mathbf{x}) = 0 $ at the inlet ($s = 0$) (i.e., property (iii)), and rearranging the terms, we get: 
{\small
\begin{align}
    \int_{\Omega} d \, \mathrm{grad}[\eta] \cdot \mathbf{K}(\mathbf{x}) 
    \mathrm{grad}[\eta] \, \mathrm{d} \Omega 
    + \int_{\Omega} h_{T} \, \eta^{2}(\mathbf{x}) \, \mathrm{d} \Omega 
    + \frac{\chi}{2} \, \eta^{2}(\mathbf{x}) \Big|_{s = 1}
    \leq \int_{\Omega} h_{T} \, \eta(\mathbf{x}) 
    \big(\vartheta_{\mathrm{amb}} - \Phi_{\min}\big) \, \mathrm{d} \Omega 
    \leq 0 
\end{align}}
In establishing the above inequality, we have used property (i), $h_T > 0$, and $\vartheta_{\mathrm{amb}} \geq \Phi_{\min}$ (cf. Eq.~\eqref{Eqn:ROM_Qualitative_definition_Phi_min}). By invoking that $d > 0$, $h_T > 0$, $\chi \geq 0$, and $\mathbf{K}(\mathbf{x})$ is positive definite, we conclude:
{\small
\begin{align}
    \label{Eqn:ROM_Qualitative_Min_Principle_derivation_Step_2}
    \int_{\Omega} d \, \mathrm{grad}[\eta] \cdot \mathbf{K}(\mathbf{x}) 
    \mathrm{grad}[\eta] \, \mathrm{d} \Omega 
    + \int_{\Omega} h_{T} \, \eta^{2}(\mathbf{x}) \, \mathrm{d} \Omega 
    + \frac{\chi}{2} \, \eta^{2}(\mathbf{x}) \Big|_{s = 1}
    = 0 
\end{align}}
The left side of the above equation is a norm, defined in terms of the function $\eta(\mathbf{x})$. Since a norm can be zero if and only if the function itself is zero, we have: 
\begin{align}
    \label{Eqn:ROM_Qualitative_Min_Principle_derivation_Step_3}
    \eta(\mathbf{x}) = 0 \quad \forall \mathbf{x} \in \overline{\Omega} 
\end{align}

Appealing to the definition of $\eta(\mathbf{x})$ (i.e., Eq.~\eqref{Eqn:ROM_Qualitative_definition_eta}) in conjunction with the above equation, we conclude that 
\begin{align}
    \label{Eqn:ROM_Qualitative_Min_Principle_derivation_Step_4}
    \Phi_{\min} 
    \leq \vartheta(\mathbf{x}) 
    \quad \forall \mathbf{x} \in \overline{\Omega}
\end{align}
The definition of $\Phi_{\min}$ (i.e., Eq.~\eqref{Eqn:ROM_Qualitative_definition_Phi_min}) renders the desired result:
\begin{align}
    \label{Eqn:ROM_Qualitative_Min_Principle_derivation_Step_5}
    \min\Big[\vartheta_{\mathrm{amb}},\vartheta_{\mathrm{inlet}}, \min_{\mathrm{x} \in \Gamma^{\vartheta}}\big[ \vartheta_{\mathrm{p}}(\mathbf{x})\big]\Big] 
    \leq \vartheta(\mathbf{x}) 
    \quad \forall \mathbf{x} \in \overline{\Omega}
\end{align}
\end{proof} 

\begin{corollary}[Non-negative solutions]
    \label{Cor:ROM_Qualitative_Nonnegative_solutions}
    Let $\vartheta(\mathbf{x}) \in C^{1}(\Omega\setminus\Sigma) \cap C^{0}(\overline{\Omega})$ satisfies weakly (i.e., in the sense of Galerkin formulation) the following differential inequalities: 
    \begin{subequations}
    \begin{alignat}{2}
    \label{Eqn:ROM_Qualitative_BoE_general_NN}
    &d \, \mathrm{div}[\mathbf{q}(\mathbf{x})] + h_{T}(\mathbf{x}) \, \vartheta(\mathbf{x}) \geq 0  
    && \quad \mathrm{in} \; \Omega \setminus \Sigma \\
    &\llbracket\mathbf{q}(\mathbf{x})\rrbracket = \chi \, \mathrm{grad}[\vartheta] \cdot \widehat{\mathbf{t}}(\mathbf{x}) 
    && \quad \mathrm{on} \; \Sigma \\
    &\llbracket\vartheta(\mathbf{x})\rrbracket = \boldsymbol{0}
    && \quad \mathrm{on} \; \Sigma \\
    &d \, \mathbf{q}(\mathbf{x}) \cdot \widehat{\mathbf{n}}(\mathbf{x}) 
    \leq 0 
    && \quad \mathrm{on} \; \Gamma^{q}  \\
    &\vartheta(\mathbf{x}) \geq 0 
    && \quad \mathrm{on} \; \Gamma^{\vartheta} \\
    \label{Eqn:ROM_Qualitative_inlet_general_NN} 
    &\vartheta \geq 0 
    && \quad s = 0 \; \mathrm{on} \; \Sigma  
\end{alignat}
\end{subequations}
under the Fourier model (i.e., $\mathbf{q}(\mathbf{x}) = -\mathbf{K}(\mathbf{x}) \, \mathrm{grad}[\vartheta]$). If $h_T(\mathbf{x}) \geq 0$, the solution field is non-negative: 
    \begin{align}
    \label{Eqn:ROM_Qualitative_NN_statement}
    0 \leq \vartheta(\mathbf{x}) 
    \quad \forall \mathbf{x} \in \overline{\Omega}
    \end{align}
\end{corollary}
\begin{proof}
Compared to the original boundary value problem \eqref{Eqn:ROM_Qualitative_BoE_general}--\eqref{Eqn:ROM_Qualitative_inlet_general}, we have herein $\vartheta_{\mathrm{amb}} = 0$, $f(\mathbf{x}) \geq 0$, $\mathrm{q}_{\mathrm{p}}(\mathbf{x}) \leq 0$, $\vartheta_{\mathrm{p}}(\mathbf{x}) \geq 0$, and $\vartheta_{\mathrm{inlet}} \geq 0$. Thus, the minimum principle (i.e., Theorem \ref{Thm:ROM_Qualitative_Minimum_principle}) immediately implies the non-negativity of the solution field: 
\begin{align}
    0 \leq 
    \min\Big[\vartheta_{\mathrm{amb}},\vartheta_{\mathrm{inlet}}, \min_{\mathrm{x} \in \Gamma^{\vartheta}}\big[ \vartheta_{\mathrm{p}}(\mathbf{x})\big]\Big] \leq \vartheta(\mathbf{x}) 
    \quad \forall \mathbf{x} \in \overline{\Omega}
\end{align}
\end{proof}

\begin{theorem}[A maximum principle] 
    \label{Thm:ROM_Qualitative_Maximum_principle}
    Let $0 \geq f(\mathbf{x}) \in L_2(\Omega)$ and $0 \leq q_{\mathrm{p}}(\mathbf{x}) \in L_2(\Gamma^{q})$ almost everywhere. Then the solution $\vartheta(\mathbf{x}) \in C^{1}(\Omega\setminus\Sigma) \cap C^{0}(\overline{\Omega})$ of the Galerkin weak formulation satisfies the following upper bound: 
    \begin{align}
    \label{Eqn:ROM_Qualitative_Maximum_principle_statement}
    \vartheta(\mathbf{x})  \leq 
    \max\Big[\vartheta_{\mathrm{amb}},\vartheta_{\mathrm{inlet}}, \max_{\mathrm{x} \in \Gamma^{\vartheta}}\big[ \vartheta_{\mathrm{p}}(\mathbf{x})\big]\Big] 
    \quad \forall \mathbf{x} \in \overline{\Omega}
    \end{align}
\end{theorem}

\begin{proof}
Let us define a new variable: 
\begin{align}
    \varphi(\mathbf{x}) = -\vartheta(\mathbf{x}) 
\end{align}
This new variable satisfies the following boundary value problem: 
\begin{subequations}
\begin{alignat}{2}
    \label{Eqn:ROM_Qualitative_BoE_MaxP} 
    -&d \, \mathrm{div}[\mathbf{K}(\mathbf{x})\mathrm{grad}[\varphi]] = \big\{-f(\mathbf{x})\big\} - h_{T} \,  \Big(\varphi(\mathbf{x}) - 
    \big\{-\vartheta_{\mathrm{amb}}\big\} \Big)
    && \quad \mathrm{in} \; \Omega \\
    \label{Eqn:ROM_Qualitative_q_jump_condition_MaxP} 
    -&d \, \llbracket\mathbf{K}(\mathbf{x}) \mathrm{grad}[\varphi]\rrbracket = \chi \,  \mathrm{grad}[\varphi] \cdot \widehat{\mathbf{t}}(\mathbf{x}) 
    && \quad \mathrm{on} \; \Sigma \\
    \label{Eqn:ROM_Qualitative_temp_jump_condition_MaxP} 
    &\llbracket \varphi(\mathbf{x})\rrbracket = 0  
    && \quad \mathrm{on} \; \Sigma \\
    \label{Eqn:ROM_Qualitative_q_BC_MaxP} 
    -&d \, \widehat{\mathbf{n}}(\mathbf{x}) \cdot \mathbf{K}(\mathbf{x})\mathrm{grad}[\varphi] 
    = \big\{-q_{\mathrm{p}}(\mathbf{x}) \big\} 
    && \quad \mathrm{on} \; \Gamma^{q}  \\
    \label{Eqn:ROM_Qualitative_temp_BC_MaxP} 
    &\varphi(\mathbf{x}) = \big\{-\vartheta_{\mathrm{p}}(\mathbf{x})\big\} 
    && \quad \mathrm{on} \; \Gamma^{\vartheta} \\
    \label{Eqn:ROM_Qualitative_inlet_BC_MaxP} 
    &\varphi = \big\{-\vartheta_{\mathrm{inlet}} \big\} 
    && \quad \mathrm{at} \, s = 0 \, \mathrm{on} \, \Sigma  
\end{alignat}
\end{subequations}

The heat sources in the above boundary value problem satisfy $\{-f(\mathbf{x}) \} \geq 0$ and $\{-q_{\mathrm{p}}(\mathbf{x}) \} \leq 0$. So the above boundary value problem, defined in terms of $\varphi(\mathbf{x})$, satisfies the requirements of the minimum principle. Noting the terms in the curly brackets and using the minimum principle on $\varphi(\mathbf{x})$, we write: 
\begin{align}
    \min\Big[-\vartheta_{\mathrm{amb}},-\vartheta_{\mathrm{inlet}}, \min_{\mathrm{x} \in \Gamma^{\vartheta}}\big[- \vartheta_{\mathrm{p}}(\mathbf{x})\big]\Big] 
    \leq \varphi(\mathbf{x})
    \quad \forall \mathbf{x} \in \overline{\Omega}
\end{align}
which is equivalent to: 
\begin{align}
    -\max\Big[\vartheta_{\mathrm{amb}},\vartheta_{\mathrm{inlet}}, \max_{\mathrm{x} \in \Gamma^{\vartheta}}\big[ \vartheta_{\mathrm{p}}(\mathbf{x})\big]\Big] 
    \leq \varphi(\mathbf{x}) \equiv 
    - \vartheta(\mathbf{x})
    \quad \forall \mathbf{x} \in \overline{\Omega}
\end{align}
By reversing the sign on both sides, which reverses the inequality, we get the desired result.
\end{proof} 

The above minimum and maximum principles render easily three other properties: a comparison principle, uniqueness of solutions, and continuous dependence of solutions on the prescribed inputs. We now state and prove these three new properties. 

\begin{theorem}[A comparison principle]
    \label{Thm:ROM_Qualitative_CP}
    For a given vasculature and ambient temperature, let 
    $\vartheta^{(1)}(\mathbf{x})$ and $\vartheta^{(2)}(\mathbf{x})$ are the temperature fields under the prescribed inputs: 
    $\left(\vartheta^{(1)}_{\mathrm{inlet}},\vartheta^{(1)}_{\mathrm{p}}(\mathbf{x}),f^{(1)}(\mathbf{x}),q_{\mathrm{p}}^{(1)}(\mathbf{x})\right)$ and  $\left(\vartheta^{(2)}_{\mathrm{inlet}},\vartheta^{(2)}_{\mathrm{p}}(\mathbf{x}),f^{(2)}(\mathbf{x}),q_{\mathrm{p}}^{(2)}(\mathbf{x})\right)$, respectively. 
    If $\vartheta^{(1)}_{\mathrm{inlet}} \leq \vartheta^{(2)}_{\mathrm{inlet}}$, $\vartheta^{(1)}_{\mathrm{p}}(\mathbf{x}) \leq \vartheta^{(2)}_{\mathrm{p}}(\mathbf{x})$, $f^{(1)}(\mathbf{x}) \leq f^{(2)}(\mathbf{x})$, and  $q_{\mathrm{p}}^{(1)}(\mathbf{x}) \geq q_{\mathrm{p}}^{(2)}(\mathbf{x})$ then 
    \begin{align}
        \vartheta^{(1)}(\mathbf{x}) \leq \vartheta^{(2)}(\mathbf{x}) \quad 
        \forall \mathbf{x} \in \overline{\Omega}
    \end{align}
\end{theorem}
\begin{proof} Let us consider the difference of the two solutions: 
\begin{align}
    \widetilde{\vartheta}(\mathbf{x}) = \vartheta^{(1)}(\mathbf{x}) - \vartheta^{(2)}(\mathbf{x})
\end{align}
Noting the ambient temperature is the same for both the cases (i.e., $\widetilde{\vartheta}_{\mathrm{amb}} = \widetilde{\vartheta}^{(1)}_{\mathrm{amb}} - \widetilde{\vartheta}^{(2)}_{\mathrm{amb}}= 0$), 
the difference $\widetilde{\vartheta}(\mathbf{x})$ satisfies the differential inequalities \eqref{Eqn:ROM_Qualitative_BoE_general_NN}--\eqref{Eqn:ROM_Qualitative_inlet_general} of Corollary \ref{Cor:ROM_Qualitative_Nonnegative_solutions}. 
The same corollary implies the non-negativity of $\widetilde{\vartheta}(\mathbf{x})$. Hence, $\vartheta^{(2)}(\mathbf{x}) \leq \vartheta^{(1)}(\mathbf{x})$ for all $\mathbf{x} \in \overline{\Omega}$. 
\end{proof}

\begin{theorem}[Uniqueness of solutions] 
\label{Thm:ROM_Qualitative_Uniqueness}
The solutions to the boundary value problem under the reduced-order model, given by Eqs.~\eqref{Eqn:ROM_Qualitative_BoE_general}--\eqref{Eqn:ROM_Qualitative_inlet_general}, are unique. 
\end{theorem}
\begin{proof}
On the contrary, assume that $\vartheta^{(1)}(\mathbf{x})$ and $\vartheta^{(2)}(\mathbf{x})$ are two solutions of the boundary value problem \eqref{Eqn:ROM_Qualitative_BoE_general}--\eqref{Eqn:ROM_Qualitative_inlet_general}. Define \begin{align}
    \widetilde{\vartheta}(\mathbf{x}) = \vartheta^{(1)}(\mathbf{x}) - \vartheta^{(2)}(\mathbf{x})
\end{align}
The difference, $\widetilde{\vartheta}(\mathbf{x})$, satisfies the following boundary value problem:
\begin{subequations}
\begin{alignat}{2}
    \label{Eqn:ROM_Qualitative_BoE_Uniqueness} 
    -&d \, \mathrm{div}[\mathbf{K}(\mathbf{x})\mathrm{grad}[\widetilde{\vartheta}]] = - h_{T} \widetilde{\vartheta}(\mathbf{x}) 
    && \quad \mathrm{in} \; \Omega \\
    \label{Eqn:ROM_Qualitative_q_jump_condition_Uniqueness} 
    -&d \, \llbracket\mathbf{K}(\mathbf{x}) \mathrm{grad}[\widetilde{\vartheta}]\rrbracket = \chi \,  \mathrm{grad}[\widetilde{\vartheta}] \cdot \widehat{\mathbf{t}}(\mathbf{x}) 
    && \quad \mathrm{on} \; \Sigma \\
    \label{Eqn:ROM_Qualitative_temp_jump_condition_Uniqueness} 
    &\llbracket\widetilde{\vartheta}(\mathbf{x})\rrbracket = 0  
    && \quad \mathrm{on} \; \Sigma \\
    \label{Eqn:ROM_Qualitative_q_BC_Uniqueness} 
    -&d \, \widehat{\mathbf{n}}(\mathbf{x}) \cdot \mathbf{K}(\mathbf{x})\mathrm{grad}[\widetilde{\vartheta}] = 0 
    && \quad \mathrm{on} \; \Gamma^{q}  \\
    \label{Eqn:ROM_Qualitative_temp_BC_Uniqueness} 
    &\widetilde{\vartheta}(\mathbf{x}) = 0 
    && \quad \mathrm{on} \; \Gamma^{\vartheta} \\
    \label{Eqn:ROM_Qualitative_inlet_BC_Uniqueness} 
    &\widetilde{\vartheta} = 0  
    && \quad \mathrm{at} \, s = 0 \, \mathrm{on} \, \Sigma  
\end{alignat}
\end{subequations}
The above boundary value problem \eqref{Eqn:ROM_Qualitative_BoE_Uniqueness}--\eqref{Eqn:ROM_Qualitative_inlet_BC_Uniqueness} resembles the original one \eqref{Eqn:ROM_Qualitative_BoE_general}--\eqref{Eqn:ROM_Qualitative_inlet_general} but with the values $f(\mathbf{x}) = 0$, $q^{\mathrm{p}}(\mathbf{x}) = 0$, $\vartheta_{\mathrm{p}}(\mathbf{x}) = 0$, $\vartheta_{\mathrm{inlet}} = 0$, and $\vartheta_{\mathrm{amb}} = 0$.
Noting these values and applying the minimum and maximum principles (i.e., Theorems \ref{Thm:ROM_Qualitative_Minimum_principle} and \ref{Thm:ROM_Qualitative_Maximum_principle}) to the boundary value problem \eqref{Eqn:ROM_Qualitative_BoE_Uniqueness}--\eqref{Eqn:ROM_Qualitative_inlet_BC_Uniqueness}, we get:
\begin{align}
    \label{Eqn:ROM_Qualitative_Inequality_Uniqueness}
    0 = \min\left[\vartheta_{\mathrm{amb}},\vartheta_{\mathrm{inlet}}, \min_{\mathbf{x} \in \Gamma^{\vartheta}} \big[ \vartheta_{\mathrm{p}}(\mathbf{x})\big] \right]
    \leq \widetilde{\vartheta}(\mathbf{x}) \leq
    \max\left[\vartheta_{\mathrm{amb}},\vartheta_{\mathrm{inlet}}, \max_{\mathbf{x} \in \Gamma^{\vartheta}} \big[ \vartheta_{\mathrm{p}}(\mathbf{x})\big] \right] = 0 
    \quad \forall \mathbf{x} \in \overline{\Omega}
\end{align}
In the above equation, the first inequality stems from the minimum principle while the second one is from the maximum principle. The equality on the either sides of Eq.~\eqref{Eqn:ROM_Qualitative_Inequality_Uniqueness} is a consequence of  $\vartheta_{\mathrm{p}}(\mathbf{x}) = 0$, $\vartheta_{\mathrm{inlet}} = 0$, and $\vartheta_{\mathrm{amb}} = 0$: which define the boundary value problem that governs $\widetilde{\vartheta}(\mathbf{x})$.

Hence, $\widetilde{\vartheta}(\mathbf{x}) = \vartheta^{(1)}(\mathbf{x}) - \vartheta^{(2)}(\mathbf{x}) = 0$ for all $\mathbf{x} \in \overline{\Omega}$, thereby establishing the uniqueness of solutions. 
\end{proof}

\begin{corollary}[Stability of solutions]
    \label{Cor:ROM_Qualitative_Stability}
    The solution fields continuously depend on $\vartheta_{\mathrm{p}}(\mathbf{x})$, $\vartheta_{\mathrm{inlet}}$, and $\vartheta_{\mathrm{amb}}$. 
\end{corollary}
\begin{layperson}
    Small changes to the prescribed inputs---the prescribed temperature on the boundary, the inlet temperature, or the ambient temperature---will lead to small changes in the temperature.
\end{layperson}
\begin{proof}
Let $(\vartheta^{*}_{\mathrm{amb}},\vartheta^{*}_{\mathrm{input}},\vartheta^{*}_{\mathrm{p}}(\mathbf{x}))$ be the altered inputs that deviate from the original prescribed inputs $(\vartheta_{\mathrm{amb}},\vartheta_{\mathrm{input}},\vartheta_{\mathrm{p}}(\mathbf{x}))$ by small changes. Mathematically, 
\begin{align}
    \big\vert \vartheta^{*}_{\mathrm{amb}} - \vartheta_{\mathrm{amb}} \big\vert < \delta_1, \; 
    \big\vert \vartheta^{*}_{\mathrm{input}} - \vartheta_{\mathrm{input}} \big\vert < \delta_2 
    \; \mathrm{and} \; 
    \max_{\mathbf{x} \in \Gamma^{\vartheta}}\big\vert \vartheta^{*}_{\mathrm{p}}(\mathbf{x}) - \vartheta_{\mathrm{p}}(\mathbf{x}) \big\vert < \delta_3
\end{align}
where $\delta_1$, $\delta_2$ and $\delta_3$ are small real numbers. Let 
\begin{align}
    \delta := \max\big[\delta_1, \delta_2, \delta_3\big] 
\end{align}
Clearly, $\delta$ is also small. 

Let $\vartheta(\mathbf{x})$ and $\vartheta^{*}(\mathbf{x})$ are, respectively, solutions of the reduced-order model with the original and altered inputs. The difference $\widetilde{\vartheta}(\mathbf{x}) := \vartheta^{*}(\mathbf{x}) - \vartheta(\mathbf{x})$ satisfies the following boundary value problem:
\begin{subequations}
\begin{alignat}{2}
    \label{Eqn:ROM_Qualitative_BoE_Stability} 
    -&d \, \mathrm{div}[\mathbf{K}(\mathbf{x})\mathrm{grad}[\widetilde{\vartheta}]] = - h_{T} \left( \widetilde{\vartheta}(\mathbf{x}) 
    - \big(\vartheta^{*}_{\mathrm{amb}} - \vartheta_{\mathrm{amb}}\big) \right) 
    && \quad \mathrm{in} \; \Omega \\
    \label{Eqn:ROM_Qualitative_q_jump_condition_Stability} 
    -&d \, \llbracket\mathbf{K}(\mathbf{x}) \mathrm{grad}[\widetilde{\vartheta}]\rrbracket = \chi \,  \mathrm{grad}[\widetilde{\vartheta}] \cdot \widehat{\mathbf{t}}(\mathbf{x}) 
    && \quad \mathrm{on} \; \Sigma \\
    \label{Eqn:ROM_Qualitative_temp_jump_condition_Stability} 
    &\llbracket\widetilde{\vartheta}(\mathbf{x})\rrbracket = 0  
    && \quad \mathrm{on} \; \Sigma \\
    \label{Eqn:ROM_Qualitative_q_BC_Stability} 
    -&d \, \widehat{\mathbf{n}}(\mathbf{x}) \cdot \mathbf{K}(\mathbf{x})\mathrm{grad}[\widetilde{\vartheta}] = 0 
    && \quad \mathrm{on} \; \Gamma^{q}  \\
    \label{Eqn:ROM_Qualitative_temp_BC_Stability} 
    &\widetilde{\vartheta}(\mathbf{x}) = 
    \vartheta^{*}_{\mathrm{p}}(\mathbf{x}) - \vartheta_{\mathrm{p}}(\mathbf{x}) 
    && \quad \mathrm{on} \; \Gamma^{\vartheta} \\
    \label{Eqn:ROM_Qualitative_inlet_BC_Stability} 
    &\widetilde{\vartheta} = \vartheta^{*}_{\mathrm{inlet}} - \vartheta_{\mathrm{inlet}}  
    && \quad \mathrm{at} \, s = 0 \, \mathrm{on} \, \Sigma  
\end{alignat}
\end{subequations}

Since the applied heat in the domain and the prescribed heat flux on the boundary are zero (cf. Eqs.~\eqref{Eqn:ROM_Qualitative_BoE_Stability}--\eqref{Eqn:ROM_Qualitative_inlet_BC_Stability} and Eqs.~\eqref{Eqn:ROM_Qualitative_BoE_general}--\eqref{Eqn:ROM_Qualitative_inlet_general}, and note $f(\mathbf{x})$ and $q_{\mathrm{p}}(\mathbf{x})$), both minimum and maximum principles hold for $\widetilde{\vartheta}(\mathbf{x})$. By applying these two principles, we get 
\begin{align}
-\delta 
&< \min\big[
\vartheta^{*}_{\mathrm{amb}} - \vartheta_{\mathrm{amb}}, \vartheta^{*}_{\mathrm{inlet}} - \vartheta_{\mathrm{inlet}}, 
\min_{\mathbf{x} \in \Gamma^{\vartheta}} [\vartheta^{*}_{\mathrm{p}}(\mathbf{x}) - \vartheta_{\mathrm{p}}(\mathbf{x})] 
\big] 
\leq \widetilde{\vartheta}(\mathbf{x}) \nonumber \\ 
&\qquad \qquad \qquad \qquad \qquad \qquad 
\leq \max\big[
\vartheta^{*}_{\mathrm{amb}} - \vartheta_{\mathrm{amb}}, \vartheta^{*}_{\mathrm{inlet}} - \vartheta_{\mathrm{inlet}}, 
\max_{\mathbf{x} \in \Gamma^{\vartheta}} [\vartheta^{*}_{\mathrm{p}}(\mathbf{x}) - \vartheta_{\mathrm{p}}(\mathbf{x})] 
\big]
< \delta 
\end{align}
We thus have  
\begin{align}
\sup_{\mathbf{x} \in \overline{\Omega}} \, 
\big\vert
\vartheta^{*}(\mathbf{x}) -
\vartheta(\mathbf{x})
\big\vert
\leq \delta 
\end{align}
which means that the corresponding change to the temperature solution field is small.
\end{proof}

The above results are general: the material properties are allowed to vary spatially, the thermal conductivity can be anisotropic, the inlet temperature can be different from the ambient temperature, and the prescribed heat sources can vary spatially (i.e., $f(\mathbf{x})$ and $q_{\mathrm{p}}(\mathbf{x})$). It is straightforward to extend the above-given results even when the ambient temperature varies spatially (i.e., $\vartheta(\mathbf{x})$). In such situations, the statements in minimum and maximum principles (i.e., Eqs.~\eqref{Eqn:ROM_Qualitative_Minimum_principle_statement} and \eqref{Eqn:ROM_Qualitative_Maximum_principle_statement}), respectively, should read: 
\begin{subequations}
\begin{align}
    &\min\Big[\min_{\mathrm{x} \in \Omega} \big[\vartheta_{\mathrm{amb}}(\mathbf{x})\big],\vartheta_{\mathrm{inlet}}, \min_{\mathrm{x} \in \Gamma^{\vartheta}}\big[ \vartheta_{\mathrm{p}}(\mathbf{x})\big]\Big] \leq \vartheta(\mathbf{x}) 
    \quad \forall \mathbf{x} \in \overline{\Omega} \\
    &\vartheta(\mathbf{x}) \leq \min\Big[\max_{\mathrm{x} \in \Omega} \big[\vartheta_{\mathrm{amb}}(\mathbf{x})\big],\vartheta_{\mathrm{inlet}}, \max_{\mathrm{x} \in \Gamma^{\vartheta}}\big[ \vartheta_{\mathrm{p}}(\mathbf{x})\big]\Big] 
    \quad \forall \mathbf{x} \in \overline{\Omega}
    \end{align}
\end{subequations} 
However, the constant ambient temperature assumption is adequate for many engineering applications. 

\section{AN EXTENSION TO INCLUDE NONLINEAR RADIATION}
\label{Sec:S4_ROM_Qualitative_Radiation}

As mentioned earlier, the inclusion of radiation makes the mathematical model nonlinear. Nonlinear models do not always have unique solutions: which will be the case even for the reduced-order model considered in this paper. One needs to impose additional conditions to achieve uniqueness. We will now look at three simpler cases, which will suggest the additional restriction that needs to be placed in the mathematical analysis for establishing maximum and comparison principles and uniqueness of solutions under thermal regulation with radiation.  

\emph{First}, let us consider the simplest setting: pure radiation under a constant heat source (i.e., no convection, conduction and active cooling/heating due to sequestered fluids in the vasculature). The problem under the mentioned setting reads: 
\begin{align}
    f_0 - \varepsilon \, \sigma \, 
    \left(\vartheta^{4} - \vartheta^{4}_{\mathrm{amb}}\right) = 0 
\end{align}
where $f_0 \geq 0$ is the constant heat source. The above (algebraic) equation has four roots (i.e., solutions), amongst which only \emph{one positive} root, given by the following expression: 
\begin{align}
    \vartheta 
    = \sqrt[4]{ \frac{f_0}{\varepsilon \, \sigma} + \vartheta^{4}_{\mathrm{amb}}}
\end{align}
The other roots comprise two complex and one negative. 

\emph{Second}, we include convection in the above problem. The governing equation then reads: 
\begin{align}
    f_0 - h_T(\vartheta - \vartheta_{\mathrm{amb}}) - \varepsilon \, \sigma \, 
    \left(\vartheta^{4} - \vartheta^{4}_{\mathrm{amb}}\right) = 0 
\end{align}
which can rewritten as the following polynomial in terms of $\vartheta$:
\begin{align}
    \label{Eqn:ROM_Qualitative_Descrates}
    \varepsilon \, \sigma \, \vartheta^{4} 
    + h_T \, \vartheta 
    - \left(f_0 + h_T \,  \vartheta_{\mathrm{amb}} 
    + \varepsilon \, \sigma \, 
     \vartheta^{4}_{\mathrm{amb}}
     \right)
    = 0 
\end{align}
Gauss' fundamental theorem of algebra implies that the above polynomial  \eqref{Eqn:ROM_Qualitative_Descrates} has four roots (in the \emph{field} of complex numbers) \citep[page 172]{brown2009complex}. Moreover, the Gauss-Descartes's rule of signs implies that this polynomial has only \emph{one positive} solution, as there is only one algebraic sign change in the (real) coefficients of the polynomial, which is ordered in the decreasing powers of the variable $\vartheta$ \citep{curtiss1918recent}. 

\emph{Third}, we consider heat transfer from conduction, convection, and radiation---but not fluid-sequestered thermal regulation. The corresponding governing equation reads:
\begin{subequations}
\begin{align}
    -d \, \mathrm{div}[\mathbf{K}(\mathbf{x})\mathrm{grad}[\vartheta]] = f(\mathbf{x}) 
    - h_{T} (\vartheta(\mathbf{x}) - \vartheta_{\mathrm{amb}}) - \varepsilon \, \sigma \,  (\vartheta^{4}(\mathbf{x}) - \vartheta^{4}_{\mathrm{amb}}) 
\end{align}
\end{subequations}
This equation is a well-studied semi-linear PDE, and the theory of partial differential equations establishes that it has a \emph{unique non-negative} solution \citep[chapter 13]{mcowen1996partial}. 

The three simpler cases discussed above reveal that unless an additional condition---such as non-negativity of solutions---is imposed, one cannot expect unique solutions under radiation. Thus, it makes sense to establish maximum and comparison principles and uniqueness for non-negative solutions, instead of all possible solutions, for the general reduced-order model considered in this paper. The mathematical requirement of non-negativity for the temperature field does not limit its applicability to physical problems, as the absolute temperature is a positive quantity. 

To this end, we again appeal to the Galerkin weak formulation, which when the radiation is included, reads: Find $\vartheta(\mathbf{x}) \in \mathcal{U}$ such that we have
\begin{align}
    \label{Eqn:ROM_Qualitative_Galerkin_weak_form_Radiation}
    \int_{\Omega} d \, \mathrm{grad}[w] \cdot \mathbf{K}(\mathbf{x}) 
    \mathrm{grad}[\vartheta] \, \mathrm{d} \Omega 
    &+ \int_{\Omega} h_{T} \, w(\mathbf{x}) (\vartheta(\mathbf{x}) - 
    \vartheta_{\mathrm{amb}}) \, \mathrm{d} \Omega 
    \nonumber \\
    &+ \int_{\Omega} \varepsilon \, \sigma \, w(\mathbf{x}) \Big(\vartheta^{4}(\mathbf{x}) - 
    \vartheta^{4}_{\mathrm{amb}}\Big) \, \mathrm{d} \Omega 
    \nonumber \\
    &+ \int_{\Sigma} \chi \, w(\mathbf{x}) \mathrm{grad}[\vartheta] \cdot \widehat{\mathbf{t}}(\mathbf{x}) \, \mathrm{d} \Gamma \nonumber \\
    &= \int_{\Omega} w(\mathbf{x}) f(\mathbf{x}) \, \mathrm{d} \Omega 
    - \int_{\Gamma^{q}} w(\mathbf{x}) q_{\mathrm{p}}(\mathbf{x}) \, \mathrm{d} \Gamma 
    \quad \forall w(\mathbf{x}) \in \mathcal{W} 
\end{align}

Following the usual strategy for nonlinear problems, we will first establish a comparison principle, which will furnish us with a minimum/maximum principle. 

\begin{theorem}[A comparison principle considering radiation]
    \label{Thm:ROM_Qualitative_Comparison_principle_Radiation}
    For a given vasculature and ambient temperature, let 
    $\vartheta^{(1)}(\mathbf{x})$ and $\vartheta^{(2)}(\mathbf{x})$ are the \emph{non-negative} temperature fields under the prescribed inputs: 
    $\left(\vartheta^{(1)}_{\mathrm{inlet}},\vartheta^{(1)}_{\mathrm{p}}(\mathbf{x}),f^{(1)}(\mathbf{x}),q_{\mathrm{p}}^{(1)}(\mathbf{x})\right)$ and  $\left(\vartheta^{(2)}_{\mathrm{inlet}},\vartheta^{(2)}_{\mathrm{p}}(\mathbf{x}),f^{(2)}(\mathbf{x}),q_{\mathrm{p}}^{(2)}(\mathbf{x})\right)$, respectively. 
    If $\vartheta^{(1)}_{\mathrm{inlet}} \leq \vartheta^{(2)}_{\mathrm{inlet}}$, $\vartheta^{(1)}_{\mathrm{p}}(\mathbf{x}) \leq \vartheta^{(2)}_{\mathrm{p}}(\mathbf{x})$, $f^{(1)}(\mathbf{x}) \leq f^{(2)}(\mathbf{x})$, and  $q_{\mathrm{p}}^{(1)}(\mathbf{x}) \geq q_{\mathrm{p}}^{(2)}(\mathbf{x})$ then 
    \begin{align}
        \vartheta^{(1)}(\mathbf{x}) \leq \vartheta^{(2)}(\mathbf{x}) \quad 
        \forall \mathbf{x} \in \overline{\Omega}
    \end{align}
\end{theorem}
\begin{proof}
Noting the radiation component is 
\begin{align}
    q_{\mathrm{rad}}(\vartheta(\mathbf{x})) = \varepsilon \, \sigma \left(\vartheta^{4}(\mathbf{x}) 
    - \vartheta_{\mathrm{amb}}^4\right) 
\end{align}
we consider the following difference: 
\begin{align}
    q_{\mathrm{rad}}(\vartheta^{(1)}(\mathbf{x}))
    - q_{\mathrm{rad}}(\vartheta^{(2)}(\mathbf{x})) 
    = \varepsilon \, \sigma \left(\Big(\vartheta^{(1)}(\mathbf{x})\Big)^{4}
   - \Big(\vartheta^{(2)}(\mathbf{x})\Big)^{4}\right) 
\end{align}
In obtaining the above expression, we have used that the ambient temperature is the same for the two cases. The mean value theorem implies that there exists $v(\mathbf{x})$ with 
\begin{align}
v(\mathbf{x}) = (1 - \alpha) \, \vartheta^{(1)}(\mathbf{x}) 
+ \alpha \, \vartheta^{(2)}(\mathbf{x})  
\end{align}
for some $\alpha \in [0,1]$ such that 
\begin{align}
     q_{\mathrm{rad}}(\vartheta^{(1)}(\mathbf{x}))
    - q_{\mathrm{rad}}(\vartheta^{(2)}(\mathbf{x})) 
    = \frac{\mathrm{d} q_{\mathrm{rad}}}{\mathrm{d}\vartheta}\Big\vert_{\vartheta(\mathbf{x}) = v(\mathbf{x})} 
    (\vartheta^{(1)}(\mathbf{x}) - \vartheta^{(2)}(\mathbf{x})) 
    =4 \, \varepsilon \, \sigma \, v^{3}(\mathbf{x}) (\vartheta^{(1)}(\mathbf{x}) - \vartheta^{(2)}(\mathbf{x})) 
\end{align}
Since $\vartheta^{(1)}(\mathbf{x})$ and $\vartheta^{(2)}(\mathbf{x})$ are non-negative (part of the hypothesis), we have 
\begin{align}
v(\mathbf{x}) \geq 0 \quad \forall \mathbf{x} \in \Omega 
\end{align}
For convenience, we define
\begin{align}
    \widetilde{h}_{T}(\mathbf{x}) = h_T + \varepsilon \, \sigma v^{3}(\mathbf{x})
\end{align}
Clearly, $\widetilde{h}_T(\mathbf{x}) \geq 0$, as $h_T \geq 0$ and $v(\mathbf{x}) \geq 0$. 

The difference in the solution fields  
\begin{align}
\widetilde{\vartheta}(\mathbf{x}) = 
\vartheta^{(2)}(\mathbf{x}) - 
\vartheta^{(1)}(\mathbf{x})
\end{align}
satisfies the following equation: 
\begin{align}
    -\mathrm{div}\Big[\mathbf{K}(\mathbf{x}) \,  \mathrm{grad}[\widetilde{\vartheta}]\Big]
    + \widetilde{h}_{T}(\mathbf{x}) \widetilde{\vartheta}(\mathbf{x}) 
    = f^{(2)}(\mathbf{x}) - f^{(1)}(\mathbf{x})
\end{align}
The above differential operator is linear with respect to $\widetilde{\vartheta}(\mathbf{x})$. The mean value theorem allows us to convert a non-linear PDE into a linear one, thereby enabling us to avail the results from the previous section. 

Noting that $f^{(1)}(\mathbf{x}) \leq f^{(2)}(\mathbf{x})$ and the other orderings for the input, this difference $\widetilde{\vartheta}(\mathbf{x})$ satisfies the differential inequalities of Corollary \ref{Cor:ROM_Qualitative_Nonnegative_solutions} with $\widetilde{h}_{T}(\mathbf{x})$ instead of $h_T$. Thus, the same corollary (i.e., non-negative solutions) renders the desired result: 
\begin{align}
\vartheta^{(2)}(\mathbf{x}) - \vartheta^{(1)}(\mathbf{x}) = 
\widetilde{\vartheta}(\mathbf{x}) \geq 0 \quad \forall \mathbf{x} \in \overline{\Omega} 
\end{align}
\end{proof}

\begin{theorem}[A minimum principle considering radiation]
    \label{Thm:ROM_Qualitative_Minimum_principle_Radiation}
    Let 
    $\vartheta(\mathbf{x}) \in C^{1}(\Omega\setminus\Sigma) \cap C^{0}(\overline{\Omega})$ be a \emph{non-negative} solution of the Galerkin weak formulation 
    under $f(\mathbf{x}) \in L_2(\Omega)$ and $q_{\mathrm{p}}(\mathbf{x}) \in L_2(\Gamma^{q})$.
    If 
    \[
    f(\mathbf{x}) \geq 0 \quad \mathrm{a.e.} 
    \quad \mathrm{and} \quad 
    q_{\mathrm{p}}(\mathbf{x}) \leq 0
    \quad \mathrm{a.e.} 
    \]
    then the solution field $\vartheta(\mathbf{x})$ satisfies the following lower bound: 
    \begin{align}
    \min\Big[\vartheta_{\mathrm{amb}},\vartheta_{\mathrm{inlet}}, \min_{\mathrm{x} \in \Gamma^{\vartheta}}\big[ \vartheta_{\mathrm{p}}(\mathbf{x})\big]\Big] \leq \vartheta(\mathbf{x}) 
    \quad \forall \mathbf{x} \in \overline{\Omega}
    \end{align}
\end{theorem}
\begin{proof}
To use the comparison principle (Theorem \ref{Thm:ROM_Qualitative_Comparison_principle_Radiation}), we take 
\begin{align}
\vartheta^{(1)}(\mathbf{x}) = \Phi_{\mathrm{min}} = \min\Big[\vartheta_{\mathrm{amb}},\vartheta_{\mathrm{inlet}}, \min_{\mathrm{x} \in \Gamma^{\vartheta}}\big[ \vartheta_{\mathrm{p}}(\mathbf{x})\big]\Big]
\end{align}
and $\vartheta^{(2)}(\mathbf{x}) = \vartheta(\mathbf{x})$. $\vartheta^{(1)}(\mathbf{x})$ is a solution for the input set:
\begin{align}
\left(\vartheta^{(1)}_{\mathrm{inlet}} = \Phi_{\mathrm{min}},\vartheta^{(1)}_{\mathrm{p}}(\mathbf{x})= \Phi_{\mathrm{min}},f^{(1)}(\mathbf{x})=0,q_{\mathrm{p}}^{(1)}(\mathbf{x})=0\right)
\end{align}
Noting that $\vartheta^{(2)}_{\mathrm{inlet}} \geq \Phi_{\mathrm{min}}$, $\vartheta^{(2)}_{\mathrm{p}}(\mathbf{x}) \geq \Phi_{\mathrm{min}}$, $f^{(2)}(\mathbf{x}) = f(\mathbf{x}) \geq 0$, $q^{(2)}_{\mathrm{p}}(\mathbf{x}) = q_{\mathrm{p}}(\mathbf{x}) \leq 0$, the comparison principle (Theorem \ref{Thm:ROM_Qualitative_Comparison_principle_Radiation}) gives the desired result:
\begin{align}
\min\Big[\vartheta_{\mathrm{amb}},\vartheta_{\mathrm{inlet}}, \min_{\mathrm{x} \in \Gamma^{\vartheta}}\big[ \vartheta_{\mathrm{p}}(\mathbf{x})\big]\Big] = \vartheta^{(1)}(\mathbf{x}) 
\leq \vartheta^{(2)}(\mathbf{x}) = \vartheta(\mathbf{x})
    \quad \forall \mathbf{x} \in \overline{\Omega}
\end{align}
\end{proof}

\begin{theorem}[A maximum principle considering radiation]
    \label{Thm:ROM_Qualitative_Maximum_principle_Radiation}
    Let 
    $\vartheta(\mathbf{x}) \in C^{1}(\Omega\setminus\Sigma) \cap C^{0}(\overline{\Omega})$ be a \emph{non-negative} solution of the Galerkin weak formulation 
    under $f(\mathbf{x}) \in L_2(\Omega)$ and $q_{\mathrm{p}}(\mathbf{x}) \in L_2(\Gamma^{q})$.
    If 
    \[
    f(\mathbf{x}) \leq 0 \quad \mathrm{a.e.} 
    \quad \mathrm{and} \quad 
    q_{\mathrm{p}}(\mathbf{x}) \geq 0
    \quad \mathrm{a.e.} 
    \]
    then the solution field $\vartheta(\mathbf{x})$ satisfies the following upper bound: 
    \begin{align}
    \vartheta(\mathbf{x}) \leq 
    \max\Big[\vartheta_{\mathrm{amb}},\vartheta_{\mathrm{inlet}}, \max_{\mathrm{x} \in \Gamma^{\vartheta}}\big[ \vartheta_{\mathrm{p}}(\mathbf{x})\big]\Big] 
    \quad \forall \mathbf{x} \in \overline{\Omega}
    \end{align}
\end{theorem}
\begin{proof}
To use the comparison principle (Theorem \ref{Thm:ROM_Qualitative_Comparison_principle_Radiation}), we take $\vartheta^{(1)}(\mathbf{x}) = \vartheta(\mathbf{x})$ and 
\begin{align}
\vartheta^{(2)}(\mathbf{x}) = \Phi_{\mathrm{max}} = \max\Big[\vartheta_{\mathrm{amb}},\vartheta_{\mathrm{inlet}}, \max_{\mathrm{x} \in \Gamma^{\vartheta}}\big[ \vartheta_{\mathrm{p}}(\mathbf{x})\big]\Big]
\end{align}
$\vartheta^{(2)}(\mathbf{x})$ is a solution for the input set:
\begin{align}
\left(\vartheta^{(2)}_{\mathrm{inlet}} = \Phi_{\mathrm{max}},\vartheta^{(2)}_{\mathrm{p}}(\mathbf{x})= \Phi_{\mathrm{max}},f^{(2)}(\mathbf{x})=0,q_{\mathrm{p}}^{(2)}(\mathbf{x})=0\right)
\end{align}
Noting that $\vartheta^{(1)}_{\mathrm{inlet}} \leq \Phi_{\mathrm{max}}$, $\vartheta^{(1)}_{\mathrm{p}}(\mathbf{x}) \geq \Phi_{\mathrm{max}}$, $f^{(1)}(\mathbf{x}) = f(\mathbf{x}) \leq 0$, $q^{(1)}_{\mathrm{p}}(\mathbf{x}) = q_{\mathrm{p}}(\mathbf{x}) \geq 0$, the comparison principle (Theorem \ref{Thm:ROM_Qualitative_Comparison_principle_Radiation}) gives the desired result:
\begin{align}
\vartheta(\mathbf{x}) = \vartheta^{(1)}(\mathbf{x}) 
\leq \vartheta^{(2)}(\mathbf{x}) = 
\max\Big[\vartheta_{\mathrm{amb}},\vartheta_{\mathrm{inlet}}, \max_{\mathrm{x} \in \Gamma^{\vartheta}}\big[ \vartheta_{\mathrm{p}}(\mathbf{x})\big]\Big]
    \quad \forall \mathbf{x} \in \overline{\Omega}
\end{align}
\end{proof}

Establishing uniqueness and stability of solutions are technical for nonlinear problems. Hence, we will not pursue them in this paper. 

\section{SPECIAL CASE: CONSTANT HEAT SOURCE AND ADIABATIC BOUNDARIES} 
\label{Sec:S5_ROM_Qualitative_Special_case}

\subsection{Useful definitions}
The \emph{mean temperature} is defined as follows: 
\begin{align}
    \label{Eqn:ROM_Sensitivity_mean_temp}
    \vartheta_{\mathrm{mean}} := \frac{1}{\mathrm{meas}(\Omega)}\int_{\Omega} 
    \vartheta(\mathbf{x}) \, \mathrm{d} \Omega 
\end{align}
where $\mathrm{meas}(\Omega)$ denotes the (set) measure of $\Omega$. Since $\Omega$ is a surface, $\mathrm{meas}(\Omega)$ is the area of the domain $\Omega$.

The \emph{hot steady-state} (HSS) refers to the steady-state achieved by the system under no flow of coolant within the vasculature (i.e., $\dot{m} = 0$). We denote the temperature at the hot steady-state by $\vartheta_{\mathrm{HSS}}$. By appealing to the balance of energy, the rate of heat supplied should be equal to the rate of heat lost due to cooling, as we have adiabatic lateral boundaries. This implies that, under constant heat source $f(\mathbf{x}) = f_0$, the temperature at HSS, denoted by $\vartheta_{\mathrm{HSS}}$, satisfies: 
\begin{align}
    \label{Eqn:ROM_Sensitivity_THSS_definition}
    h_{T} \, (\vartheta_{\mathrm{HSS}} - \vartheta_{\mathrm{amb}}) + \varepsilon \, \sigma \, (\vartheta_{\mathrm{HSS}}^4 - \vartheta_{\mathrm{amb}}^4) = f_0 
\end{align}
Under pure convection (without radiation), $\vartheta_{\mathrm{HSS}}$ can be written as: 
\begin{align}
    \label{Eqn:ROM_Sensitivity_THSS_definition_convection}
    \vartheta_{\mathrm{HSS}} = \vartheta_{\mathrm{amb}} + \frac{1}{h_T} f_{0}
\end{align}
Clearly, for a heat source $f_{0} \geq 0$, we have 
\begin{align}
    \vartheta_{\mathrm{amb}} \leq \vartheta_{\mathrm{HSS}}
\end{align}
Also, $\vartheta_{\mathrm{HSS}}$ is independent of $\mathbf{x}$ for a constant $f(\mathbf{x})$.

Then, for a constant heat source, we define the \emph{active-cooling} to be the case when $\vartheta_{\mathrm{input}} \leq \vartheta_{\mathrm{HSS}}$, and \emph{active-heating} occurs when $\vartheta_{\mathrm{input}} \geq \vartheta_{\mathrm{HSS}}$.

\subsection{Special case} We will now consider a special case with the following assumptions: 
\begin{enumerate}
    \item[(A1)] The heat source is constant: $f(\mathbf{x}) = f_0$. 
    \item[(A2)] The entire lateral boundary is adiabatic (i.e., $\partial \Omega = \Gamma^{q}$ and $q_{\mathrm{p}}(\mathbf{x}) = 0$). 
    \item[(A3)] The inlet temperature is lower than the hot steady-state (i.e., $\vartheta_{\mathrm{inlet}} \leq \vartheta_{\mathrm{HSS}}$): active-cooling. 
\end{enumerate}

The motivation behind this special case is that many experiments and numerical simulations reported in the literature have used such conditions, for example, \citep{devi2021microvascular,pejman2019gradient}. These studies have also assumed $\vartheta_{\mathrm{inlet}} = \vartheta_{\mathrm{amb}}$ besides the above three assumptions. However, we do not use this fourth condition in the analysis presented in this section, as the properties we derive are unaffected with this additional condition if the above three conditions are met. Some of these properties, in the form of bounds, are manifest: a direct consequence of the principles discussed in the previous section. While some others are not, as they are valid only under the restrictive assumptions of the special case.

\subsubsection{Temperature is bounded below by $\vartheta_{\mathrm{inlet}}$ and above by $\vartheta_{\mathrm{HSS}}$}
Noting that $\Gamma^{\vartheta} = \emptyset$, the minimum principle (Theorem \ref{Thm:ROM_Qualitative_Minimum_principle}) provides the following lower bound: 
\begin{align}
    \label{Eqn:ROM_Qualitative_special_case_To_wit}
    \min\big[\vartheta_{\mathrm{amb}}, \vartheta_{\mathrm{inlet}}\big]
    \leq \vartheta(\mathbf{x}) 
    \quad \forall \mathbf{x} \in \overline{\Omega}
\end{align}
But invoking further that the heat source is uniform will render a tighter lower bound and an upper bound; the next proposition establishes these bounds.

\begin{proposition}
Under the special case (i.e., assumptions A1--A3 hold), the temperature field satisfies the following bounds:
\begin{align}
    \vartheta_{\mathrm{inlet}}
    \leq \vartheta(\mathbf{x}) 
    \leq \vartheta_{\mathrm{HSS}}
    \quad \forall \mathbf{x} \in \overline{\Omega}
\end{align}
\end{proposition}
\begin{proof}
    We will first establish the lower bound. If $\vartheta_{\mathrm{inlet}} <  \vartheta_{\mathrm{amb}}$, then the lower bound is manifest from Eq.~\eqref{Eqn:ROM_Qualitative_special_case_To_wit}. For the case $\vartheta_{\mathrm{amb}} \leq \vartheta_{\mathrm{inlet}}$, we invoke the comparison principle (Theorem \ref{Thm:ROM_Qualitative_Comparison_principle_Radiation}) by taking $\vartheta^{(1)}(\mathbf{x}) = \vartheta_{\mathrm{inlet}}$ and $\vartheta^{(2)}(\mathbf{x}) = \vartheta(\mathbf{x})$. $\vartheta^{(1)}(\mathbf{x}) =  \vartheta_{\mathrm{inlet}}$ is the solution for the input set: 
    \begin{align}
    \left(\vartheta^{(1)}_{\mathrm{inlet}} = \vartheta_{\mathrm{inlet}},f^{(1)}(\mathbf{x}) = f_0 - h_T(\vartheta_{\mathrm{inlet}} - \vartheta_{\mathrm{amb}}) - \varepsilon \, \sigma (\vartheta^{4}_{\mathrm{inlet}} - \vartheta^{4}_{\mathrm{amb}}),q_{\mathrm{p}}^{(1)}(\mathbf{x})=0\right)
    \end{align}
    Since $\vartheta_{\mathrm{amb}} \leq \vartheta_{\mathrm{inlet}}$, $f^{(1)}(\mathbf{x}) \leq f_0$. We also note that  $\vartheta^{(2)}_{\mathrm{inlet}} = \vartheta_{\mathrm{inlet}}$, $f^{(2)}(\mathbf{x}) = f_0$, $q^{(2)}_{\mathrm{p}}(\mathbf{x}) = q_{\mathrm{p}}(\mathbf{x}) = 0$. Then comparison principle implies that 
    \begin{align}
    \vartheta_{\mathrm{inlet}} 
    = \vartheta^{(1)}(\mathbf{x}) 
    \leq \vartheta^{(2)}(\mathbf{x}) 
    = \vartheta(\mathbf{x}) 
    \quad \forall \mathbf{x} \in \overline{\Omega}
    \end{align}
    
    For the upper bound, we will again invoke comparison principle by taking $\vartheta^{(1)}(\mathbf{x}) = \vartheta(\mathbf{x})$ and $\vartheta^{(2)}(\mathbf{x}) = \vartheta_{\mathrm{HSS}}$. The latter is a solution for the input set:
    \begin{align}
\left(\vartheta^{(2)}_{\mathrm{inlet}} = \Phi_{\mathrm{HSS}},f^{(2)}(\mathbf{x})=f_0,q_{\mathrm{p}}^{(2)}(\mathbf{x})=0\right)
\end{align}
Note that, in writing the above input set, $\vartheta_{\mathrm{p}}(\mathbf{x})$ in not included, as the entire boundary under the special case is prescribed with the heat flux (i.e., $\Gamma^{\vartheta} = \emptyset$). 
Further noting that $\vartheta^{(1)}_{\mathrm{inlet}} = \Phi_{\mathrm{inlet}} \leq \Phi_{\mathrm{HSS}}$, $f^{(1)}(\mathbf{x}) = f_0$, $q^{(1)}_{\mathrm{p}}(\mathbf{x}) = q_{\mathrm{p}}(\mathbf{x}) = 0$, the comparison principle (Theorem \ref{Thm:ROM_Qualitative_Comparison_principle_Radiation}) gives the desired upper bound:
\begin{align}
    \vartheta(\mathrm{x}) 
    = \vartheta^{(1)}(\mathbf{x}) 
    \leq \vartheta^{(2)}(\mathbf{x}) 
    = \vartheta_{\mathrm{HSS}}
    \quad \forall \mathbf{x} \in \overline{\Omega}
\end{align}
\end{proof}
    
\subsubsection{Mean temperature is bounded below by $\vartheta_{\mathrm{inlet}}$ and above by $\vartheta_{\mathrm{HSS}}$}

This result is a direct consequence of the previous result: if the entire temperature field lies between $\vartheta_{\mathrm{inlet}}$ and $\vartheta_{\mathrm{HSS}}$, the mean temperature should also respect the bounds. An alternative proof, without appealing to the comparison principle, is also provided in Appendix \ref{ROM_Qualitative_Appendix}. 

\subsubsection{Outlet temperature is bounded below by $\vartheta_{\mathrm{inlet}}$}
Integrating the governing equation \eqref{Eqn:ROM_Qualitative_BoE_general} over the domain, using the divergence theorem and the jump condition \eqref{Eqn:ROM_Qualitative_q_jump_condition_general}, and invoking the definition of $\vartheta_{\mathrm{HSS}}$ (i.e., Eq.~\eqref{Eqn:ROM_Sensitivity_THSS_definition}), we obtain the following equation: 
\begin{align}
    \chi \, (\vartheta_{\mathrm{outlet}} - \vartheta_{\mathrm{inlet}}) 
    = \int_{\Omega} h_{T} \, (\vartheta_{\mathrm{HSS}} - \vartheta(\mathbf{x})) \,  \mathrm{d}\Omega 
    + \int_{\Omega} \varepsilon 
    \, \sigma \, 
    \Big(\vartheta^{4}_{\mathrm{HSS}} - \vartheta^{4}(\mathbf{x})\Big) 
    \,\mathrm{d} \Omega 
\end{align}
Since the temperature field is bounded above by $\vartheta_{\mathrm{HSS}}$ and $\chi \geq 0$, we establish:
\begin{align}
    \vartheta_{\mathrm{inlet}} \leq \vartheta_{\mathrm{outlet}}
\end{align}

\begin{tcolorbox}
For the special case considered in this section, if $\vartheta_{\mathrm{inlet}} \leq \vartheta_{\mathrm{HSS}}$, irrespective of the magnitude of the ambient temperature, we have the following orderings:
\begin{align}
    &\vartheta_{\mathrm{inlet}} 
    \leq \vartheta(\mathbf{x}) 
    \leq \vartheta_{\mathrm{HSS}} 
    \quad \forall \mathbf{x} \\ 
    &\vartheta_{\mathrm{inlet}} 
    \leq \vartheta_{\mathrm{mean}} 
    \leq \vartheta_{\mathrm{HSS}} \\
    \label{Eqn:ROM_Special_case_inlet_LT_outlet}
    &\vartheta_{\mathrm{inlet}} \leq 
    \vartheta_{\mathrm{outlet}}
\end{align}
\end{tcolorbox}

\section{ON THE POSSIBILITY OF $\vartheta_{\mathrm{inlet}} \geq \vartheta_{\mathrm{outlet}}$} 
\label{Sec:S6_ROM_TOutlet_LT_TInlet}

Given the ordering of the inlet and outlet temperatures under the special case (i.e., Eq.~\eqref{Eqn:ROM_Special_case_inlet_LT_outlet}), a natural question arises: Is it possible to have $\vartheta_{\mathrm{inlet}} \geq \vartheta_{\mathrm{outlet}}$ under a heat source (i.e., $f(\mathbf{x}) \geq 0$) and with adiabatic lateral boundaries?

Intuition gained from prior studies (which have used the conditions similar to the special case discussed in \S\ref{Sec:S5_ROM_Qualitative_Special_case} or $\vartheta_{\mathrm{inlet}} = \vartheta_{\mathrm{amb}}$) might suggest the answer to the above question is negative. However,  we show that the answer to this question is affirmative. We use mathematical analysis to construct a corresponding scenario logically. 

We first note from \S\ref{Sec:S5_ROM_Qualitative_Special_case} that if the heat source is uniform throughout the domain (i.e., $f(\mathbf{x}) = f_0 > 0$ in $\Omega$) then 
\begin{align}
    \vartheta_{\mathrm{input}} \leq \vartheta(\mathbf{x}) \leq 
    \vartheta_{\mathrm{HSS}} \quad \forall \mathbf{x} \in \overline{\Omega} 
\end{align}
irrespective of the ordering of the ambient and inlet temperatures (i.e., whether $\vartheta_{\mathrm{inlet}} \geq \vartheta_{\mathrm{amb}}$ or not). Hence, for the mentioned case $\vartheta_{\mathrm{input}} \leq \vartheta_{\mathrm{output}}$.
Next, we note that for a general heat source---not necessarily constant, $\vartheta_{\mathrm{inlet}} \leq \vartheta_{\mathrm{amb}}$, and adiabatic lateral boundaries (i.e., $\Gamma^{\vartheta} = \emptyset$), the minimum principle (Theorem \ref{Thm:ROM_Qualitative_Minimum_principle_Radiation}) implies: 
\begin{align}
    \vartheta_{\mathrm{input}} = \min[\vartheta_{\mathrm{amb}},\vartheta_{\mathrm{input}}] \leq \vartheta(\mathbf{x}) \quad \forall \mathbf{x} \in \overline{\Omega} 
\end{align}
Thus, we once again have $ \vartheta_{\mathrm{input}} \leq 
    \vartheta_{\mathrm{output}}$.
So, to construct a simple counterexample, we need to break the conditions in the two cases just discussed: select a non-uniform heat source with $\vartheta_{\mathrm{inlet}} \geq \vartheta_{\mathrm{amb}}$.

\textbf{Figure~\ref{Fig:ROM_TOutlet_LT_TInlet}A} provides one such boundary value problem: the heat source applied only on a portion of the domain---to achieve non-uniformity. We solved this problem numerically using the weak form  capability available in \citet[version 5.6]{COMSOL}.  \textbf{Table~\ref{Table:ROM_Qualitative_Simulation_parameters}} provides the parameters used in the numerical simulation, with $\vartheta_{\mathrm{amb}} = 298.15 \; \mathrm{K}$ lower than $\vartheta_{\mathrm{inlet}} = 315 \; \mathrm{K}$. \textbf{Figure~\ref{Fig:ROM_TOutlet_LT_TInlet}B} shows the associated numerical result using the finite element method with the Galerkin weak formulation (i.e., Eq.~\eqref{Eqn:ROM_Qualitative_Galerkin_weak_form_Radiation}). The import of this numerical result is multi-fold: 
\begin{enumerate}
    \item It verifies the maximum and comparison principles and illustrates their utility. 
    \item It also manifests clearly a ramification of the inlet temperature differing from the ambient temperature. For this problem (Fig.~\ref{Fig:ROM_TOutlet_LT_TInlet}), the outlet temperature $\vartheta_{\mathrm{outlet}} = 312.4 \; \mathrm{K}$ is lesser than the inlet temperature $\vartheta_{\mathrm{inlet}} = 315 \; \mathrm{K}$. If the same problem is solved with the inlet temperature smaller than or equal to the ambient temperature, the minimum principle (Theorem \ref{Thm:ROM_Qualitative_Minimum_principle_Radiation}) implies the reverse ordering: the outlet temperature is greater than the inlet temperature. 
    \item Thus, it disproves a commonly held belief in thermal regulation that the outlet temperature is higher than the inlet temperature if the heat is supplied to the body by the heater (i.e., when the source is non-negative: $f(\mathbf{x}) \geq 0$ in the entire domain). 
\end{enumerate}

\begin{table}[h]
\caption{This table provides the parameters used in the numerical simulation. These parameters are taken from the literature \citep{devi2021microvascular,pejman2019gradient}; these values correspond to the glass-fiber-reinforced (GFRP) composite as the host solid and water as the coolant. \label{Table:ROM_Qualitative_Simulation_parameters}}
\begin{tabular}{ll}\hline 
Length of the domain $L$ & 100 mm \\
Height of the domain $H$ & 100 mm \\
Thickness of the body $d$ & 4.31 mm \\
Isotropic conductivity of the host solid $k$ & 0.5593 $\mathrm{W/m/K}$ \\
Applied heater power $f_0$ & 500 $\mathrm{W/m^2}$ \\ 
Convective heat transfer coefficient $h_T$ & 13 $\mathrm{W/m^2/K}$ \\ 
Emissivity $\varepsilon$ & 0.95 \\
Stefan-Boltzmann constant $\sigma$ & $5.67 \times 10^{-8}$ $\mathrm{W/m^2/K^4}$ \\ 
Ambient temperature $\vartheta_{\mathrm{amb}}$ & 295.15 K 
(22 $^\circ$C) \\
Inlet temperature $\vartheta_{\mathrm{Inlet}}$ & 315 K \\ 
Specific heat capacity of the fluid $c_f$ & 4183 $\mathrm{J/kg/K}$ \\
Density of the fluid $\rho_f$ & 1000 $\mathrm{kg/m^3}$ \\ 
Volumetric flow rate $Q$ & $11.564 \times 10^{-3}$ $\mathrm{L/min}$ \\ 
Mass flow rate $\dot{m}$ & $\dot{m} = \rho_{f} \, Q = 11.564 \times 10^{-3}$ $\mathrm{kg/min}$ \\ 
Heat capacity rate $\chi$ & $\chi = \dot{m} \, c_f = 48.372$ $\mathrm{J/min/K}$ \\ 
\hline 
\end{tabular}
\end{table}

\begin{figure}[h]
    \centering
    \includegraphics[scale=1.65]{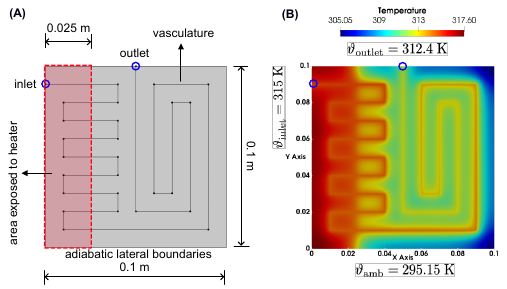}
    \caption{This figure illustrates one of the main findings of this paper: the outlet temperature can be lower than the inlet temperature (i.e.,  $\vartheta_{\mathrm{outlet}} \leq \vartheta_{\mathrm{inlet}}$) even under a heat source. (A) A pictorial description of the boundary value problem, showing the vasculature and the locations of the inlet and outlet. Only one-fourth of the domain is exposed to a constant heat source as indicated. All the lateral boundaries are adiabatic. (B) The temperature field in the domain is shown, indicating the inlet and outlet temperatures.
    \label{Fig:ROM_TOutlet_LT_TInlet}}
\end{figure}

Given that the inlet and outlet temperatures can be comparatively smaller or larger, several scientific inquiries on the nature of the temperature field along the vasculature follow naturally. One such fundamental question is:  \emph{Does the temperature always vary monotonically (either increasing or decreasing) along the vasculature?} This answered question warrants further research. 

\section{APPROPRIATE DEFINITION FOR EFFICIENCY}
\label{Sec:S7_ROM_Qualitative_Efficiency}

In prior studies on active cooling, the \emph{cooling efficiency} is defined as the ratio of the rate of heat extracted by the coolant to the total heat supplied by the heat source (heater); for example, see \citep{devi2021microvascular}. That is, 
\begin{align}
     \label{Eqn:ROM_standard_efficiency_general}
    \eta^{e} := \frac{\mbox{rate of heat extracted by the coolant}}{\mbox{rate of heat supplied by the heater}}
\end{align}
where $\eta^{e}$ denotes the cooling efficiency. If the lateral boundaries are adiabatic ($\Gamma^{q} = \partial \Omega$ and $q_{\mathrm{p}}(\mathbf{x}) = 0$) then $\eta^{e}$ takes the following mathematical form:
\begin{align}
    \label{Eqn:ROM_standard_efficiency}
    \eta^{e} 
    &= \frac{\int_{\Sigma} \chi \, \mathrm{grad}[\vartheta]\cdot \widehat{\mathbf{t}}(\mathbf{x}) \, \mathrm{d} \Gamma}{\int_{\Omega} f(\mathbf{x}) \, \mathrm{d} \Omega}
    = \chi \,
    \left(\int_{\Omega} f(\mathbf{x}) \, \mathrm{d} \Omega \right)^{-1}
    \left(\vartheta_{\mathrm{outlet}} - \vartheta_{\mathrm{inlet}}\right) 
\end{align}
where $\vartheta_{\mathrm{outlet}}$ is the temperature at the outlet. 

\subsection{Applicability}
Many prior studies, including the one cited above, have assumed a constant heat source $f(\mathbf{x}) = f_0$ and $\vartheta_{\mathrm{amb}} = \vartheta_{\mathrm{inlet}}$. Under such conditions, the above definition for efficiency renders values between 0 and 1. To wit, in Section \ref{Sec:S6_ROM_TOutlet_LT_TInlet}, we have shown that $\vartheta_{\mathrm{outlet}} \geq \vartheta_{\mathrm{inlet}}$ when the entire lateral boundary is adiabatic and the heat source is a constant in the entire domain, thereby rendering $\eta^{e} \geq 0$. To show the upper bound, we use the energy balance in the entire domain to write the following: 
\begin{align}
    \eta^{e} 
    &= 1 - \frac{\mbox{rate of heat lost due to convection + rate of heat lost due to radiation}}{\mbox{rate of heat supplied by the heater}} \notag \\
    &= 1 - \frac{1}{f_0 \, \mathrm{meas}(\Omega)} 
    \left[
    \int_{\Omega}h_T (\vartheta(\mathbf{x}) - \vartheta_{\mathrm{amb}}) \, \mathrm{d} \Omega
    + \int_{\Omega}\varepsilon \, \sigma \, (\vartheta^{4}(\mathbf{x}) - \vartheta^{4}_{\mathrm{amb}}) \, \mathrm{d} \Omega
    \right]
\end{align}
Since $\vartheta_{\mathrm{amb}} \leq \vartheta(\mathbf{x})$ (on the account of the minimum principle when $\vartheta_{\mathrm{inlet}} = \vartheta_{\mathrm{amb}}$) and noting that $f_0 \geq 0$, $\mathrm{meas}(\Omega) > 0$, $h_T \geq 0$, $\varepsilon \geq 0$, and $\sigma > 0$, we have $\eta^{e} \leq 1$. 

However, the above definition (Eq.~\eqref{Eqn:ROM_standard_efficiency_general} or \eqref{Eqn:ROM_standard_efficiency}) is not applicable in general. Three such scenarios are:
\begin{enumerate}
    \item When $\vartheta_{\mathrm{inlet}} >  \vartheta_{\mathrm{amb}}$, from the previous section, we know that $\vartheta_{\mathrm{inlet}} \geq \vartheta_{\mathrm{outlet}}$ for a non-uniform heat source. Hence, the efficiency based on Eq.~\eqref{Eqn:ROM_standard_efficiency} will be negative. 
    \item When $\vartheta_{\mathrm{inlet}} < \vartheta_{\mathrm{amb}}$, the efficiency given by  Eq.~\eqref{Eqn:ROM_standard_efficiency} can exceed the unity. \textbf{Figure~\ref{Fig:Efficiency_eta_GT_unity}} presents a numerical result, using the finite element method, verifying the efficiency $\eta^{e}$ exceeding one. (The Galerkin weak formulation \eqref{Eqn:ROM_Qualitative_Galerkin_weak_form_Radiation} and \citet{COMSOL} are used to generate even this numerical result.)
\item Of course, the above definition \eqref{Eqn:ROM_standard_efficiency} is not applicable for active heating; if used, $\eta^{e}$ will be negative. 
\end{enumerate}
Hence, for studies on active-cooling thermal regulation with $\vartheta_{\mathrm{inlet}} \neq \vartheta_{\mathrm{amb}}$, we suggest that $\eta^{e}$ be referred to as the \emph{coefficient of performance}, rather than cooling efficiency.

\begin{figure}[h]
    \centering
    \includegraphics[scale=1.65]{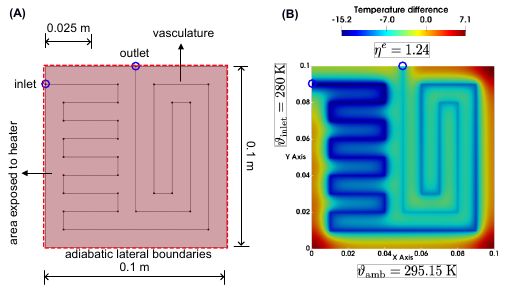}
    \caption{This figure shows that the widely used definition for efficiency could render a value greater than unity. (A) A pictorial description of the boundary value problem. The parameters used in this simulation are the same as provided in Table \ref{Table:ROM_Qualitative_Simulation_parameters} except for $\vartheta_{\mathrm{inlet}} = 280 \; \mathrm{K}$ and the applied heat flux $f_0 = 500$ $\mathrm{W/m^2}$ is over the entire domain. The calculated efficiency is $\eta^{e} = 1.24$ (cf. Eq.~\eqref{Eqn:ROM_standard_efficiency_general}). (B) The profile of the \emph{difference} between the spatial temperature field and the ambient temperature (i.e., $\vartheta(\mathbf{x}) - \vartheta_{\mathrm{amb}}$) is shown. There are regions where this difference is negative, implying that the heat flows into the body from the ambient space. At the same time, there are other regions where this difference is positive. It so happens that the net heat transfer is into the system from the ambient space. Given that the lateral boundaries are adiabatic, energy balance implies that the flowing fluid extracts this heat from the ambient besides the heat from the heater, rendering the efficiency $\eta^{e}$ greater than the unity.  \label{Fig:Efficiency_eta_GT_unity}}
\end{figure}

\subsection{Alternative definitions}
Due to the mentioned deficiencies in the currently used definition of efficiency for thermal regulation, we define alternatives, guided by the mathematical principles derived in the previous sections. We restrict these definitions to a constant heat source (i.e., $f(\mathbf{x}) = f_0$ in $\Omega$). 

\subsubsection{Active cooling ($\vartheta_{\mathrm{inlet}} \leq \vartheta_{\mathrm{HSS}}$)} A candidate for \emph{cooling efficiency}---the efficiency under active cooling---is:  
\begin{align}
    \label{Eqn:ROM_Qualitative_eta_active_cooling}
    \eta^{e}_{\mathrm{cooling}} = \frac{\vartheta_{\mathrm{HSS}} - \vartheta_{\mathrm{mean}}}{\vartheta_{\mathrm{HSS}} - \min[\vartheta_{\mathrm{inlet}},\vartheta_{\mathrm{amb}}]}
\end{align}
The comparison principle implies that, for a \emph{constant} heat source, the lower and upper bounds for the mean temperature are $\vartheta_{\mathrm{inlet}}$ and $\vartheta_{\mathrm{HSS}}$, respectively. These bounds imply that  
\begin{align}
    0 \leq \eta^{e}_{\mathrm{cooling}}
    \leq \eta^{e}_{\mathrm{max}}
\end{align}
where the upper bound $\eta^{e}_{\mathrm{max}}$---referred to as the \emph{maximum cooling efficiency}---is defined as: 
\begin{align}
    \label{Eqn:ROM_Qualitative_max_cooling_efficiency}
    \eta^{e}_{\mathrm{max}}
    := 
    \left\{\begin{array}{ll}
    1 & \mathrm{if} \; \vartheta_{\mathrm{inlet}} \leq \vartheta_{\mathrm{amb}} \\
    \frac{\vartheta_{\mathrm{HSS}} - \vartheta_{\mathrm{inlet}} }{\vartheta_{\mathrm{HSS}} - \vartheta_{\mathrm{amb}}} & 
    \mathrm{if} \; \vartheta_{\mathrm{inlet}} \geq \vartheta_{\mathrm{amb}} \\
    \end{array}\right.
\end{align}

For the problem shown in \textbf{Fig.~\ref{Fig:Efficiency_eta_GT_unity}}, we have 
$\vartheta_{\mathrm{mean}} = 288.63 \; \mathrm{K}$, $\vartheta_{\mathrm{HSS}} = 323.8 \; \mathrm{K}$,  $\vartheta_{\mathrm{input}} = 280 \; \mathrm{K}$, and $\vartheta_{\mathrm{amb}} = 295.15 \; \mathrm{K}$. For these values, the newly defined cooling efficiency amounts to: 
\begin{align}
    \eta^{e}_{\mathrm{cooling}} &=  \frac{323.8  - 288.63}{323.8 - 280} = 0.803 
\end{align}
which is more meaningful than $\eta^{e} = 1.24$. 

The proposed definition for active cooling \eqref{Eqn:ROM_Qualitative_eta_active_cooling} resembles that of a steam engine: as a ratio of two temperatures \citep{cengel2011thermodynamics}. Notably, the maximum cooling efficiency is independent of the conductivity of the host solid, flow rate within the vasculature, properties of the fluid (coolant), and the nature of the vasculature. Said differently, whatever one chooses for the mentioned quantities, the efficiency of active cooling cannot exceed the \emph{maximum cooling efficiency}. Although the maximum cooling efficiency has a similar looking as the Carnot efficiency---the maximum possible efficiency of a steam engine, both have different connotations. Carnot efficiency is the maximum possible fraction (or percentage) of the heat that can be converted into mechanical work in a steam engine \citep{atkins2010laws}. On the other hand, there is \emph{no mechanical work} involved in the definitions for the cooling efficiency \eqref{Eqn:ROM_Qualitative_eta_active_cooling} and the maximum cooling efficiency \eqref{Eqn:ROM_Qualitative_max_cooling_efficiency} under thermal regulation. 

\subsubsection{Active heating ($\vartheta_{\mathrm{inlet}} \geq \vartheta_{\mathrm{HSS}}$)} A candidate for efficiency under active heating can be: 
\begin{align}
    \label{Eqn:ROM_Qualitative_eta_active_heating}
    \eta^{e}_{\mathrm{heating}} = \frac{\vartheta_{\mathrm{mean}} - \vartheta_{\mathrm{HSS}}}{\max[\vartheta_{\mathrm{inlet}},\vartheta_{\mathrm{amb}}] - \vartheta_{\mathrm{HSS}}}
\end{align}
The comparison principle implies that $\vartheta_{\mathrm{HSS}} \leq \vartheta(\mathbf{x})$, which further implies that $\vartheta_{\mathrm{HSS}} \leq \vartheta_{\mathrm{mean}}$. The maximum principle implies that $\max[\vartheta_{\mathrm{input}},\vartheta_{\mathrm{amb}}] \geq \vartheta_{\mathrm{input}} \geq \vartheta_{\mathrm{HSS}}$. On these accounts, we have:
\begin{align}
    0 \leq \eta_{\mathrm{heating}}^{e} \leq 1
\end{align}
where the left bound is achieved when there is no flow in the vasculature (i.e., no active heating), meaning $\vartheta_{\mathrm{mean}} = \vartheta_{\mathrm{HSS}}$. 

Finally, it is still an open question: \emph{what is an appropriate measure of efficiency for thermal regulation applications when the heat source is not uniform?}  

\section{CLOSURE}
\label{Sec:S8_ROM_Qualitative_Closure}

This paper presented a reduced-order model for thermal regulation in vascular systems---applicable for active cooling and active heating. This general mathematical model allows the inlet temperature to differ from the ambient temperature, the conductivity of the host solid to be anisotropic, a body's surface to convect and radiate freely, and heat exchange via the flow of fluids within an embedded vasculature. The main contribution reveals several qualitative properties that thermal regulation in vascular systems possesses. These properties---presented as minimum, maximum, and comparison principles---place bounds on the point-wise temperature field, mean temperature, and temperature at the outlet of the vasculature. Moreover, these mathematical properties reveal the inadequacy of the currently used efficiency metric for thermal regulation, corroborated by numerical simulations as well. Precisely, the widely used definition---the ratio of heat extracted by fluid in a vasculature to the heat supplied by a source---does not always render a value between 0 and 1. These violations---uncharacteristic of a legitimate definition for efficiency---manifest when the inlet temperature differs from ambient and for active heating. Therefore, guided by the derived principles, the paper proposes alternative definitions of efficiency for both active cooling and active heating. 

The mentioned qualitative properties, bounds on the temperature, and clarification on efficiency will guide designers in building new thermal regulation systems. This newfound understanding will reduce trial and error and facilitate optimized thermal management in synthetic material systems. 

Several open research questions were peppered throughout the paper, providing adequate motivation and background. Addressing these questions will undoubtedly advance the field of thermal regulation. Two such questions are: 
\begin{enumerate}[Q1)]
\item How does the temperature field vary along the vasculature? Under what conditions is this variation monotonic? 
\item How to quantify efficiency, in general, for thermal regulation in vascular systems? 
\end{enumerate}
We also encourage an experimental study with different inlet and ambient temperatures to (a) validate the ramifications, predicted in this paper, of such a condition on the thermal regulation and (b) supplement modeling efforts in answering the two questions mentioned above.

The following quote sums up this paper's intended outcome: 
\begin{quote}
``\emph{The further a mathematical theory is developed, the more harmoniously and uniformly does its construction proceed, and unsuspected relations are disclosed between hitherto separated branches of the science.}" --- \textsc{David Hilbert}
\end{quote}

\appendix 
\section{ALTERNATIVE PROOF} 
\label{ROM_Qualitative_Appendix}

We will provide an alternative proof for establishing the mean temperature is bounded below by $\vartheta_{\mathrm{inlet}}$ and above by $\vartheta_{\mathrm{HSS}}$ for the \emph{special case} considered in Section \ref{Sec:S5_ROM_Qualitative_Special_case}. Taking a different approach than the earlier one, the proof in this appendix does not appeal to the comparison principle. Given its simplicity, this alternative proof possesses pedagogical value---accessible and beneficial to engineers and students.

Recall the assumptions under the special case (cf. \S\ref{Sec:S5_ROM_Qualitative_Special_case}): 
\begin{enumerate}
    \item The heat source is constant: $f(\mathbf{x}) = f_0$. 
    \item The entire lateral boundary is adiabatic (i.e., $\partial \Omega = \Gamma^{q}$ and $q_{\mathrm{p}}(\mathbf{x}) = 0$). 
    \item The inlet temperature is lower than the hot steady-state (i.e., $\vartheta_{\mathrm{inlet}} \leq \vartheta_{\mathrm{HSS}}$). 
\end{enumerate}

We start the proof by multiplying both sides of Eq.~\eqref{Eqn:ROM_Qualitative_BoE_general} by $(\vartheta(\mathbf{x}) - \vartheta_{\mathrm{inlet}})$ and integrating over the entire domain:  
\begin{align}
 \int_{\Omega} -d \, (\vartheta(\mathbf{x}) - \vartheta_{\mathrm{inlet}}) \, \mathrm{div}\big[\mathbf{K}(\mathbf{x}) \mathrm{grad}[\vartheta]\big] \, \mathrm{d} \Omega    
 &= \int_{\Omega} f_0 \, (\vartheta(\mathbf{x}) - \vartheta_{\mathrm{inlet}}) \, \mathrm{d} \Omega \nonumber \\
 &\qquad - \int_{\Omega} h_{T} (\vartheta(\mathbf{x}) - \vartheta_{\mathrm{inlet}}) 
 (\vartheta(\mathbf{x}) - \vartheta_{\mathrm{amb}}) \, \mathrm{d} \Omega \notag \\
 &\qquad - \int_{\Omega} \varepsilon 
 \, \sigma \, (\vartheta(\mathbf{x}) - \vartheta_{\mathrm{inlet}}) 
 \Big(\vartheta^{4}(\mathbf{x}) - \vartheta^{4}_{\mathrm{amb}}\Big) \, \mathrm{d} \Omega
\end{align}
Invoking the Green's identity on the first term and then using the jump conditions 
\eqref{Eqn:ROM_Qualitative_temp_jump_condition_general} and  \eqref{Eqn:ROM_Qualitative_q_jump_condition_general}, we get the following: 
\begin{align}
    \int_{\Omega} d \,
    \mathrm{grad}[\vartheta] 
    \cdot \mathbf{K}(\mathbf{x}) \mathrm{grad}[\vartheta] 
    \, \mathrm{d} \Omega
    &-\int_{\partial \Omega} d \, (\vartheta(\mathbf{x}) - \vartheta_{\mathrm{inlet}}) \,
    \widehat{\mathbf{n}}(\mathbf{x}) \cdot \mathbf{K}(\mathbf{x}) \mathrm{grad}[\vartheta] 
    \, \mathrm{d} \Gamma \nonumber \\
    &+ \int_{\Sigma} \chi \, (\vartheta(\mathbf{x}) - \vartheta_{\mathrm{inlet}}) \, \mathrm{grad}[\vartheta] \cdot \widehat{\mathbf{t}}(\mathbf{x}) \, \mathrm{d} \Gamma 
    \notag \\ 
    &= \int_{\Omega} f_0 \, (\vartheta(\mathbf{x}) - \vartheta_{\mathrm{inlet}}) \, \mathrm{d} \Omega 
    - \int_{\Omega} h_{T} (\vartheta(\mathbf{x}) - \vartheta_{\mathrm{inlet}})
    (\vartheta(\mathbf{x}) - \vartheta_{\mathrm{amb}})
    \, \mathrm{d} \Omega \notag \\
    &\qquad - \int_{\Omega} \varepsilon 
    \, \sigma \, (\vartheta(\mathbf{x}) - \vartheta_{\mathrm{inlet}}) 
    \Big(\vartheta^{4}(\mathbf{x}) - \vartheta^{4}_{\mathrm{amb}}\Big) \, \mathrm{d} \Omega
\end{align}
The first integral is non-negative, as the thermal conductivity is bounded below (i.e., positive definite) and $d > 0$. The second integral vanishes on the account of Eq.~\eqref{Eqn:ROM_Qualitative_q_BC_general}. Noting that $\vartheta_{\mathrm{inlet}}$ and $\chi$ are constants, we integrate along the third integral along the vasculature and simplify it as follows: 
\begin{align}
 \label{Eqn:ROM_Sensitivity_Step_3}
 \int_{\Sigma} \chi \, (\vartheta(\mathbf{x}) - \vartheta_{\mathrm{inlet}}) \, \mathrm{grad}[\vartheta] 
 \cdot \widehat{\mathbf{t}}(\mathbf{x}) \, \mathrm{d} \Gamma 
 &= \int_{\Sigma} \frac{\chi}{2} \, \mathrm{grad}\big[(\vartheta(\mathbf{x}) - \vartheta_{\mathrm{inlet}})^2\big]  
 \cdot \widehat{\mathbf{t}}(\mathbf{x}) 
 \, \mathrm{d} \Gamma \notag \\
 &= \frac{\chi}{2}
 (\vartheta_{\mathrm{outlet}} - \vartheta_{\mathrm{inlet}})^2 \geq 0 
\end{align}
Note that we have used $\chi \geq 0$ in establishing the non-negativity. 

Based on the above account, Eq.~\eqref{Eqn:ROM_Sensitivity_Step_3} provides the following inequality: 
\begin{align}
    \label{Eqn:ROM_Sensitivity_Appendix_Step_4}
    \int_{\Omega} f_0 \, (\vartheta(\mathbf{x}) - \vartheta_{\mathrm{inlet}}) \, \mathrm{d} \Omega 
    &\geq \int_{\Omega} h_{T} \,
    (\vartheta(\mathbf{x}) - \vartheta_{\mathrm{inlet}})
    (\vartheta(\mathbf{x}) - \vartheta_{\mathrm{amb}}) 
    \, \mathrm{d} \Omega \notag \\
    &\qquad + \int_{\Omega} \varepsilon \, \sigma \, (\vartheta(\mathbf{x}) - \vartheta_{\mathrm{inlet}}) \Big(\vartheta^{4}(\mathbf{x}) - \vartheta^{4}_{\mathrm{amb}}\Big) \, \mathrm{d} \Omega
\end{align}
Multiplying the both sides of  Eq.~\eqref{Eqn:ROM_Sensitivity_THSS_definition}---the  definition of $\vartheta_{\mathrm{HSS}}$---by $(\vartheta(\mathbf{x}) - \vartheta_{\mathrm{inlet}})$ and integrating over the domain, we get: 
\begin{align}
    \label{Eqn:ROM_Sensitivity_Appendix_Step_5}
    \int_{\Omega} f_0 \, (\vartheta(\mathbf{x}) - \vartheta_{\mathrm{inlet}}) \, \mathrm{d} \Omega 
    &= \int_{\Omega} h_{T} \, (\vartheta(\mathbf{x}) - \vartheta_{\mathrm{inlet}}) (\vartheta_{\mathrm{HSS}} - \vartheta_{\mathrm{amb}}) \, \mathrm{d} \Omega 
    \notag \\ 
    &\qquad + \int_{\Omega} \varepsilon \, \sigma \, (\vartheta(\mathbf{x}) - \vartheta_{\mathrm{inlet}}) \Big(\vartheta^{4}_{\mathrm{HSS}} - \vartheta^{4}_{\mathrm{amb}}\Big) \, \mathrm{d} \Omega 
\end{align}
By subtracting Eq.~\eqref{Eqn:ROM_Sensitivity_Appendix_Step_5} from Eq.~\eqref{Eqn:ROM_Sensitivity_Appendix_Step_4}, we get: 
\begin{align}
    \label{Eqn:ROM_Sensitivity_Appendix_Step_6}
    \int_{\Omega} h_{T} \,
    (\vartheta(\mathbf{x}) - \vartheta_{\mathrm{inlet}})
    (\vartheta(\mathbf{x}) - \vartheta_{\mathrm{HSS}}) 
    \, \mathrm{d} \Omega 
    + \int_{\Omega} \varepsilon \, \sigma \, (\vartheta(\mathbf{x}) - \vartheta_{\mathrm{inlet}}) \Big(\vartheta^{4}(\mathbf{x}) - \vartheta^{4}_{\mathrm{HSS}}\Big) \, \mathrm{d} \Omega
    \leq 0 
\end{align}

\subsection{Lower bound} 
We rearrange the terms in the above inequality to obtain the following: 
\begin{align}
    \label{Eqn:ROM_Sensitivity_Appendix_Step_7}
    \int_{\Omega} h_{T} \,
    (\vartheta(\mathbf{x}) - \vartheta_{\mathrm{inlet}})^{2}
    \, \mathrm{d} \Omega 
    &+ \int_{\Omega} \varepsilon \, \sigma \, (\vartheta(\mathbf{x}) - \vartheta_{\mathrm{inlet}}) \Big(\vartheta^{4}(\mathbf{x}) - \vartheta^{4}_{\mathrm{inlet}}\Big) \, \mathrm{d} \Omega
    \notag \\ 
    &\qquad \leq  
    (\vartheta_{\mathrm{HSS}} - \vartheta_{\mathrm{inlet}}) 
    \int_{\Omega} h_{T} \,
    (\vartheta(\mathbf{x}) - \vartheta_{\mathrm{inlet}})
    \, \mathrm{d} \Omega \notag \\
    &\qquad \qquad + \Big(\vartheta^{4}_{\mathrm{HSS}} - \vartheta^{4}_{\mathrm{inlet}}\Big) \int_{\Omega} \varepsilon \, \sigma \, (\vartheta(\mathbf{x}) - \vartheta_{\mathrm{inlet}}) \, \mathrm{d} \Omega
\end{align}
The second integral on the left side of the inequality is non-negative. To wit, 
\begin{align}
    \label{Eqn:ROM_Sensitivity_Appendix_Step_8}
    (\vartheta(\mathbf{x}) - \vartheta_{\mathrm{inlet}}) \Big(\vartheta^{4}(\mathbf{x}) - \vartheta^{4}_{\mathrm{inlet}}\Big) 
    = (\vartheta(\mathbf{x}) - \vartheta_{\mathrm{inlet}})^{2} 
    (\vartheta(\mathbf{x}) +  \vartheta_{\mathrm{inlet}}) \Big(\vartheta^{2}(\mathbf{x}) + \vartheta^{2}_{\mathrm{inlet}}\Big) 
    \geq 0 
\end{align}
as the temperature field and inlet temperature are non-negative. Given $\varepsilon \geq 0$ and $\sigma > 0$, the integrand is non-negative, thereby making the integral non-negative. Noting that the first integral is also non-negative, we get: 
\begin{align}
    \label{Eqn:ROM_Sensitivity_Appendix_Step_9}
    0 \leq  
    \left(h_T (\vartheta_{\mathrm{HSS}} - \vartheta_{\mathrm{inlet}}) + \varepsilon \, \sigma \Big(\vartheta^{4}_{\mathrm{HSS}} - \vartheta^{4}_{\mathrm{inlet}}\Big) \right) 
    \, \mathrm{meas}(\Omega) \, 
    (\vartheta_{\mathrm{mean}} - \vartheta_{\mathrm{inlet}})
\end{align}
Since $\vartheta_{\mathrm{inlet}} \leq \vartheta_{\mathrm{HSS}}$ and  $\mathrm{meas}(\Omega) > 0$, we have established the lower bound: $\vartheta_{\mathrm{inlet}} \leq \vartheta_{\mathrm{mean}}$. 

\subsection{Upper bound} 
We rearrange the terms in inequality \eqref{Eqn:ROM_Sensitivity_Appendix_Step_6} to obtain the following:
\begin{align}
    \label{Eqn:ROM_Sensitivity_Appendix_Step_10}
    \int_{\Omega} h_{T} \,
    (\vartheta(\mathbf{x}) - \vartheta_{\mathrm{HSS}})^{2}
    \, \mathrm{d} \Omega 
    &+ \int_{\Omega} \varepsilon \, \sigma \, (\vartheta(\mathbf{x}) - \vartheta_{\mathrm{HSS}}) \Big(\vartheta^{4}(\mathbf{x}) - \vartheta^{4}_{\mathrm{HSS}}\Big) \, \mathrm{d} \Omega
    \notag \\ 
    &\qquad \leq  
    (\vartheta_{\mathrm{HSS}} - \vartheta_{\mathrm{inlet}}) 
    \int_{\Omega} h_{T} \,
    (\vartheta_{\mathrm{HSS}} - \vartheta(\mathbf{x}))
    \, \mathrm{d} \Omega \notag \\
    &\qquad \qquad + (\vartheta_{\mathrm{HSS}} - \vartheta_{\mathrm{inlet}}) \int_{\Omega} \varepsilon \, \sigma \, \Big(\vartheta^{4}_{\mathrm{HSS}} - \vartheta^{4}(\mathbf{x})\Big) \, \mathrm{d} \Omega
\end{align}
Following a reasoning similar to the one used in establishing the lower bound, the left-hand side in the above inequality can be shown to be non-negative. Since $\vartheta_{\mathrm{inlet}} \leq  \vartheta_{\mathrm{HSS}}$, we thus have:
\begin{align}
    \label{Eqn:ROM_Sensitivity_Appendix_Step_11}
    0 \leq  
    \int_{\Omega} h_{T} \,
    (\vartheta_{\mathrm{HSS}} - \vartheta(\mathbf{x}))
    \, \mathrm{d} \Omega 
    + \int_{\Omega} \varepsilon \, \sigma \, \Big(\vartheta^{4}_{\mathrm{HSS}} - \vartheta^{4}(\mathbf{x})\Big) \, \mathrm{d} \Omega
\end{align}
Cauchy-Schwarz inequality on $L_2(\Omega)$ inner-product space implies \citep{garling2007inequalities}: 
\begin{align}
    \label{Eqn:ROM_Sensitivity_Appendix_Step_12}
    \int_{\Omega} \vartheta^{4}(\mathbf{x}) 
    \, \mathrm{d} \Omega \geq 
    \mathrm{meas}(\Omega) \;  \vartheta^{4}_{\mathrm{mean}}
\end{align}
By combining the above two inequalities we get
\begin{align}
    \label{Eqn:ROM_Sensitivity_Appendix_Step_13}
    0 \leq  
    h_T (\vartheta_{\mathrm{HSS}} - \vartheta_{\mathrm{mean}}) + \varepsilon \, \sigma \Big(\vartheta^{4}_{\mathrm{HSS}} - \vartheta^{4}_{\mathrm{mean}}\Big)
\end{align}
which furnishes us with the upper bound: $\vartheta_{\mathrm{mean}} \leq \vartheta_{\mathrm{HSS}}$. 

\section*{DATA AVAILABILITY}
The data that support the findings of this study are available from the corresponding author upon reasonable request.

\section*{ACKNOWLEDGEMENTS}
The author thanks Professor Jason F.~Patrick, North Carolina State University, for fruitful discussions on thermal regulation in synthetic composites. 

\bibliographystyle{plainnat}
\bibliography{Master_References}
\end{document}